\documentclass[11pt]{article}
\usepackage{mathtools} % Math
\usepackage{amsthm}
\usepackage{xcolor} % Colors
\usepackage{tikz} % Graphs
\usetikzlibrary{shapes.geometric, positioning, arrows}
\usepackage{todonotes} % Notes

\newtheorem{theorem}{Theorem}%[section]
\newtheorem*{theorem*}{Theorem}
\newtheorem{lemma}[theorem]{Lemma}
\newtheorem*{lemma*}{Lemma}
\newtheorem{corollary}[theorem]{Corollary}

\newtheorem{assumption}{Assumption}

\newtheorem*{definition*}{Definition}

\newtheorem{remark}{Remark}

\tikzstyle{bus} = [circle, minimum height=0.5cm, text centered, draw=black, fill=white]
\tikzstyle{generator} = [rectangle, minimum height=0.5cm, text centered, draw=black, fill=black!10]
\tikzstyle{load} = [rectangle, rounded corners, minimum height=0.5cm, text centered, draw=black, fill=black!0]
\tikzstyle{tie} = [-, line width=1mm]
\tikzstyle{line} = [-, thick]
\tikzstyle{candidate} = [dashed, thick, black!50]

\newcommand{\reals}{\mathbf{R}}

\newcommand{\diag}{\mathbf{diag}}
\newcommand{\ones}{\mathbf{1}}

\newcommand{\proj}{\mathbf{proj}}
\newcommand{\envelope}{\mathbf{env}}
\newcommand{\op}{\mathrm{op}}

\newcommand{\discrete}{\textrm{int}}
\newcommand{\cont}{\textrm{cont}}

\newcommand{\iter}{{(i)}}
\newcommand{\iterplus}{{(i+1)}}

\usepackage{parskip}
\usepackage[
    letterpaper,
    margin=1in,
]{geometry}
\usepackage[tworuled]{algorithm2e}
\usepackage{pgfplots}
\usepackage{url}
\usepackage{verbatim}

\usetikzlibrary{calc}
\pgfplotsset{compat=1.18}

\newcounter{problem}
\newcommand{\problemtag}{\stepcounter{problem}\tag{P\theproblem}}

\title{Gradient Methods for Scalable Multi-value Electricity Network Expansion Planning}
\author{
Anthony Degleris
\and
Abbas El Gamal
\and
Ram Rajagopal
}
\date{\today}

\begin{document}

\maketitle

\begin{abstract}
We consider \textit{multi-value expansion planning} (MEP), a general bilevel optimization model in which a planner optimizes arbitrary functions of the dispatch outcome in the presence of a partially controllable, competitive electricity market.
The MEP problem can be used to jointly plan various grid assets, such as transmission, generation, and battery storage capacities;
examples include identifying grid investments that minimize emissions in the absence of a carbon tax, maximizing the profit of a portfolio of renewable investments and long-term energy contracts, or reducing price inequities between different grid stakeholders.
The MEP problem, however, is in general nonconvex, making it difficult to solve exactly for large real-world systems.
Therefore, we propose a fast stochastic implicit gradient-based heuristic method that scales well to large networks with many scenarios.
We use a strong duality reformulation and the McCormick envelope to provide a lower bound on the performance of our algorithm via convex relaxation.
We test the performance of our method on a large model of the U.S.\ Western Interconnect and demonstrate that it scales linearly with network size and number of scenarios and can be efficiently parallelized on large machines.
We find that for medium-sized 16 hour cases, gradient descent on average finds a 5.3x lower objective value in 16.5x less time compared to a traditional reformulation-based approach solved with an interior point method.
We conclude with a large example in which we jointly plan transmission, generation, and storage for a 768 hour case on 100 node system, showing that emissions penalization leads to additional 40.0\% reduction in carbon intensity at an additional cost of \$17.1/MWh.
\end{abstract}

% We introduce a generalization of classical power system expansion planning problems called \textit{multi-value expansion planning} (MEP), in which the planner can optimize arbitrary functions of the dispatch outcome in the presence of a uncontrolled, competitive electricity market.

\section{Introduction}

Expanding the electric power system is essential to future grid electrification and greenhouse gas curtailment~\cite{Larson2021-xg, Joskow2020-ua, Brown2021-ks, Brown2018-yh}.
However, there are numerous barriers inherent to developing new grid assets, including high upfront capital costs, long and uncertain development timelines, and extensive regulatory restrictions~\cite{Joskow2021-zo}. 
There are also significant challenges associated with value assessment and alignment.
First, neglecting benefits outside of economic cost and reliability may under-estimate the value of a new transmission, generation, or storage project.
Second, various utilities, private organizations, and government agencies generally participate in planning grid expansions, all with competing values and objectives.

To address these challenges, modern expansion planning studies attempt to assess a range of benefits outside of cost and reliability, such as greenhouse gas emissions reductions or increased competition among generating assets.
For example, the California Independent System Operator (ISO) has developed the Transmission Economic Assessment Methodology (TEAM) to holistically evaluate transmission projects by assessing economic benefits separately for different market participants, accounting for downstream impacts on market prices and generator competition, and measuring environmental impacts~\cite{Awad2010-bt}.

Methodologies such as TEAM are effective tools for \textit{evaluating} new grid investments in a holistic, market-aware framework.
In this paper, we focus on optimization-based expansion planning models, a complementary approach for \textit{identifying} potential grid investments from a large set of candidate solutions.
Specifically, we study \textit{multi-value expansion planning} (MEP), a general optimization-based framework for expansion planning problems that allows the planner to optimize arbitrary functions of the dispatch outcome in an (partially controllable) competitive electricity market.
MEP supports jointly planning transmission, generation, and storage assets in either a centralized environment, e.g., a vertically integrated system, or as part of a decentralized market with multiple agents optimizing distinct objectives. 

We formulate multi-value expansion planning as a bilevel optimization problem, in which the upper-level problem is the planner's investment decision and the lower-level problem is market clearing process.
We give several examples of MEP problems, including minimizing carbon emissions in lieu of imposing a carbon tax or maximizing the profit of a portfolio of renewable investments and long-term energy contracts.
Using strong duality~\cite{Garces2009-xp, Zare2019-fv}, we reformulate the MEP problem as a single-level optimization problem with bilinear equality constraints.
When the dispatch costs are convex quadratic or linear, the resulting problem is a nonconvex quadratic program, which is NP-hard in general.
Therefore, inspired by its success in machine learning~\cite{LeCun2015-qc}, we propose applying gradient descent directly to the \textit{implicit form} of multi-value expansion planning and show that it converges to local stationary point in polynomial time.
We also use a simple convex relaxation of the MEP problem to bound the performance of our algorithm and generate good initial candidate solutions.
We show that our approach scales well to large problems with many scenarios and outperforms applying interior point methods to the reformulated problem.
Finally, we solve MEP problems on a large model of the U.S.\ Western Interconnect, demonstrating that gradient descent can find good solutions to realistically-sized grid planning problems.

\section{Literature review}
\label{sec:literature}

% Expansion planning is the problem of identifying optimal long-term upgrades and expansions to a resource system, e.g., electricity, water, or gas networks.
In this work, we focus on expansion planning for bulk power systems and consider resources such as generation, transmission, and storage.
Several comprehensive reviews are available on electricity grid expansion planning, usually specific to an individual resource; see, for example,~\cite{Koltsaklis2018-wx, Mahdavi2019-ev, Sheibani2018-aw}.
Expansion planning studies classically model the system as vertically integrated and minimize total system cost.
Recently, however, numerous studies directly model market behavior using hierarchical equilibrium models; see~\cite{Gonzalez-Romero2020-zp} and references therein.
Here, we adopt the latter market-based approach.
We also limit the scope of our study to single investment period expansion planning with scenario-based uncertainty, although much of our work is straightforward to extend to more general settings.

% \paragraph{Expansion planning.}
% Power system expansion planning problems have been studied in depth over the past several decades~\cite{Pereira1985-sn, Romero1994-si, Romero2002-zp, Qi2015-hw, Majidi-Qadikolai2018-hb, Neumann2019-mj}.
% The classical approach to expansion planning optimizes investment decisions and operation decisions jointly to minimize total system cost~\cite{Pereira1985-sn, Romero2002-zp};
% this most closely models the case of a vertically integrated utility with full centralized control of the power system.
% Generally, expansion planning models specialize to a particular resource type, such as transmission~\cite{Majidi-Qadikolai2018-hb, Neumann2019-mj}, generation~\cite{Pineda2016-cd, Baringo2011-xy}, or storage~\cite{Xu2018-jr}, although some models consider multiple resources simultaneously~\cite{Bent2011-dq, Asgharian2017-uj}.
% Generation and storage problems are usually (mixed-integer) convex and straightforward to solve with off-the-shelf methods~\cite{Alguacil2003-hl}.
% Problems involving transmission expansion have bilinear constraints (assuming a DC-linearized power flow model), but can be relaxed to a convex program using standard techniques~\cite{Alguacil2003-hl, Borraz-Sanchez2016-up}; the relaxations are often exact for integer investments.
% Finally, many expansion planning models have been extended to incorporate uncertainty and robustness~\cite{Zou2018-vz, Garcia-Bertrand2017-hq}.

\paragraph{Bilevel (market-based) expansion planning.}
Deregulated electricity markets make expansion planning significantly more complex because different agents have distinct, competing interests.
In this setting, the resulting expansion planning problem is formulated as bilevel optimization problem or a Stackelberg game~\cite{Sauma2006-ky, Gonzalez-Romero2020-zp}, which can be cast as a mathematical program with equilibrium constraints (MPEC).
Most bilevel expansion planning problems fall into one of two classes.
The first are \textit{proactive} expansion planning problems which maximize social welfare while anticipating strategic (profit-maximizing) generator investment and bidding behavior~\cite{Sauma2006-ky, Pozo2013-gr, Pozo2013-iz, Pozo2017-kz}.
Under perfect competition, however, no generation entity has sufficient market power to affect prices and the Stackelberg equilibrium can be computed as the result of a (standard) single-level planning problem~\cite{Garces2009-xp, Munoz2012-ge}.
The second class of problems is \textit{profit maximization problems} which consider a single or coalition of market participants choosing expansion decisions to maximize individual profit, as opposed to overall welfare~\cite{Baringo2011-xy, Maurovich-Horvat2015-qu, Akbari2016-ja, Xu2018-jr}.
Although the same reduction to single-level planning under perfect competition still applies, it is generally less relevant since considerations about competition are essential to pricing for an individual firm.
Finally, a recent work~\cite{Degleris2021-qg} considers a system operator expanding transmission and storage to reduce greenhouse gas emissions in a market environment without a carbon tax.

One contribution of our work is to provide a unified framework for solving a large class of bilevel expansion planning problems.
In particular, the implementation and performance of our gradient algorithms generally do not change significantly across a broad range of bilevel expansion planning problems, regardless of their specific problem details.

\paragraph{Solution algorithms.}
Generally, the standard approach for solving a bilevel expansion planning problem is to reduce it to a nonconvex single-level problem, then apply standard techniques for solving nonconvex programs.
There are two common approaches for this:
to write the equilibrium conditions for the lower level problem explicitly using the Karush-Kuhn-Tucker (KKT) conditions~\cite{Pineda2018-ry, Constante-Flores2022-hm}, or to use strong duality to reformulate the optimality conditions as an inequality constraint~\cite{Garces2009-xp, Garcia-Herreros2016-vk, Zhang2017-wn}.
Both approaches introduce several bilinear terms and make the resulting problem nonconvex; see~\cite{Zare2019-fv} for a comparison between the two.
In this work, we use the strong duality-based approach to derive our convex relaxation.

Our main solution algorithm, inspired by similar work on bilevel optimization in machine learning~\cite{Pedregosa2016-rr, Lorraine2019-ja, Yang2022-ep}, is to apply gradient descent \textit{directly} to the original implicit form of the bilevel problem.
This requires differentiating \textit{through} an optimization problem, which can be performed efficiently for a large class of problems using \textit{implicit differentiation}~\cite{Dontchev2014-ck, Lorraine2019-ja}.
The main advantages of our approach are that it scales well to high-dimensional and stochastic problems with many scenarios~\cite{Bubeck2015-hz}, can be easily warm started if the problem is perturbed, and is simple to implement.

\section{Multi-value expansion planning}
\label{sec:problem}

In this section, we first introduce the \textit{dispatch model}, which is used to make real-time operational decisions given long-term planning decisions, and describe classical expansion planning.
We then formulate multi-value expansion planning and show how classical planning can be viewed as a special case.
Finally, we show how to extend multi-value expansion planning to uncertain environments.

\subsection{Dispatch model}
\label{sec:dispatch}

Consider an electricity network characterized by long-term physical parameter vector $\eta \in \mathcal H \subset \reals^K$, which we refer to as the \textit{network parameters} or \textit{network investments}, which may include the values of any long-term investment decisions.
The network parameters include, for example, transmission line capacities and susceptances, generator capacities, and battery capacities.
We refer to the set $\mathcal H \subset \reals^M$ as the set of allowable network investments, which we assume to be compact.
The \textit{dispatch model} is the problem of choosing real-time operational decisions to minimize the system cost,
\begin{equation} \label{eq:dispatch}
\begin{array}{ll}
    \textrm{minimize} \quad & c(x) \\[0.25em]
    \textrm{subject to}
    & A(\eta) x \leq b(\eta),
\end{array}
\end{equation}
where the variable is $x \in \reals^N$, which specifies all short-term operation decisions, such as the power output of each generator.
The cost function $c : \reals^N \rightarrow \reals$ specifies the cost of operating the system, e.g., the total cost of producing power.
The functions $A : \reals^K \rightarrow \reals^{M \times N}$ and $b : \reals^K \rightarrow \reals^M$ describe the $M$ affine inequality constraints of the system as a function of the network parameters.
These constraints may include, for example, limits on generation and transmission capacity.
For ease of notation, we only consider problems with inequality constraints;
however, all our technical results should be straightforward to extend to dispatch models with equality constraints as well.
Throughout the text, we may refer to the \textit{operator}, the entity that solves the dispatch model and sets $x$.
We will make the following assumptions about the dispatch model.

\begin{assumption}[Strict feasibility, i.e., Slater's condition]
\label{assume:feasible}
    For all $\eta \in \mathcal H$, there exists $x \in \reals^N$ such that $A(\eta) x < b(\eta)$.
\end{assumption}

\begin{assumption}[Differentiable strongly convex costs]
\label{assume:convex-diff}
    The cost function $c$ is strongly convex and twice differentiable with Lipschitz continuous second derivatives.
\end{assumption}

\paragraph{Dispatch map.}
The first two assumptions imply that for all $\eta \in \reals^K$, problem~\eqref{eq:dispatch} is convex and has a unique solution.
Let $x^*(\eta) : \reals^K \rightarrow \reals^N$ denote the optimal choice of $x$ in~\eqref{eq:dispatch} given network parameters $\eta$.
The function $x^*$ is well-defined, since~\eqref{eq:dispatch} has a unique solution.
We call the function $x^*$ the \textit{dispatch map}, since it maps the network parameters to the solution of the dispatch model.

\paragraph{Dual problem.}
The dual problem of~\eqref{eq:dispatch} is
\begin{equation} \label{eq:dual-dispatch}
\begin{array}{ll}
    \textrm{maximize} \quad & u(\lambda, \eta) \\[0.25em]
    \textrm{subject to}
    & \lambda \geq 0,
\end{array}
\end{equation}
where the variable is $\lambda \in \reals^M$, and $u(\lambda, \eta) = -\lambda^T b(\eta) - c^\star(-A(\eta)^T \lambda)$.
Here the function $c^\star(y) = - \inf_x \left( c(x) - y^T x \right)$ is the convex conjugate of $c$, which is concave and, in many cases, has a simple analytical expression~\cite{Boyd2004-tx}.
The variable $\lambda$ is the dual variable of the constraint $A(\eta) x \leq b(\eta)$ in~\eqref{eq:dispatch}.
We make two additional assumptions that guarantee dual uniqueness.

\begin{assumption}[Non-redundant constraint]
\label{assume:non-redundant-constraints}
    Denote $\tilde a_1(\eta), \ldots, \tilde a_M(\eta) \in \reals^n$ the rows of $A(\eta)$ and let $m_1(\eta), \ldots, m_p(\eta) \in \{1, \ldots, M\}$ be the indices of binding constraints at optimality, i.e., the indices $m$ such that $\tilde a_m(\eta)^T x^*(\eta) = b(\eta)$.
    Then, for all $\eta \in \mathcal H$, the vectors $\tilde a_{m_1(\eta)}(\eta), \ldots, \tilde a_{m_p(\eta)}(\eta)$ are linearly independent.
\end{assumption}

\begin{assumption}[Strict complementarity]
\label{assume:strict-complement}
    Denote $\mathcal B(\eta) = \{ m_1(\eta), \ldots, m_p(\eta) \}$ the indices of the binding constraints at optimality, as before and let $\lambda^*$ be a solution to~\eqref{eq:dual-dispatch}.
    Then, for all $\eta \in \mathcal H$, $\lambda_m^* > 0$ for $m \in \mathcal B(\eta)$.
\end{assumption}

Assumptions~\ref{assume:convex-diff},~\ref{assume:non-redundant-constraints}, and~\ref{assume:strict-complement} together imply that the solution to~\eqref{eq:dual-dispatch} exists uniquely for all $\eta \in \reals^K$; see Appendix~\ref{appendix:unique-dual}.
We call~\eqref{eq:dual-dispatch} the \textit{dual dispatch problem} and define $\lambda^*(\eta) : \reals^K \rightarrow \reals^N$, which we call the \textit{dual dispatch map}, the optimal choice of $\lambda$ in~\eqref{eq:dual-dispatch} as a function of $\eta$.
We make two final assumptions to simplify several technical results.

\begin{assumption}[Bounded domain]
\label{assume:bounded-domain}
    There exist convex, compact sets $\mathcal X \subset \reals^N$ and $\Lambda \subset \reals^M$ such that $x^*(\eta) \in \mathcal X$ and $\lambda^*(\eta) \in \Lambda$ for all $\eta \in \mathcal H$.
\end{assumption}

\begin{assumption}[Affine coefficients]
\label{assume:affine-param}
    The functions $A$ and $b$ are affine (and hence differentiable) in $\eta$.
\end{assumption}

\paragraph{Discussion of assumptions.}
Assumption~\ref{assume:feasible} says that for every allowable network investment, the dispatch model is strictly feasible.
If we assume load curtailment is a feasible, albeit expensive, dispatch outcome, then we can generally guarantee feasibility by allowing the `zero power' solution, i.e., no power is produced or consumed by any device in the network.
Given feasibility, we can also guarantee strict feasibility by perturbing the right hand side of the inequality constraints by a small constant.
Assumptions~\ref{assume:non-redundant-constraints},~\ref{assume:strict-complement}, and~\ref{assume:bounded-domain} can also generally be guaranteed through careful problem construction and small perturbations.
Perhaps most restrictive is the strong convexity requirement in Assumption~\ref{assume:convex-diff}.
However, this too can be guaranteed through problem perturbation, e.g., adding a small quadratic penalty to each variable.
Moreover, empirical results suggest that this assumption can be relaxed, so long as the resulting problem has a unique solution for every $\eta \in \mathcal H$.
Finally, Assumption~\ref{assume:affine-param} is mainly intended to simplify various technical results and can be relaxed to just twice differentiability if needed.

\subsection{Expansion planning}

\textit{Expansion planning} is the problem of choosing both long-term investments and real-time operation decisions to minimize overall system cost,
\begin{equation} \label{eq:cep}
\begin{array}{ll}
    \textrm{minimize} \quad & \gamma^T \eta + c(x) \\[0.25em]
    \textrm{subject to}
    & A(\eta) x \leq b(\eta) \\[0.25em]
    & \eta \in \mathcal H,
\end{array}
\end{equation}
where the variables are $x \in \reals^N$ and $\eta \in \reals^K$.
The vector $\gamma \in \reals^K$ gives the cost of the network investment.
Two common types of constraints on network investments include continuous box constraints, where $\mathcal H^{\cont} = \{ \eta : \eta^{\min} \leq \eta \leq \eta^{\max} \}$ for $\eta^{\min}, \eta^{\max} \in \reals^K$, and binary box constraints, where $\mathcal H^{\discrete} = \{ \eta^{\min}_1, \eta^{\max}_1 \} \times \cdots \times \{ \eta^{\min}_k, \eta^{\max}_k \}$.
Throughout the remainder of the paper, we will consider continuous investments decisions, $\mathcal H = \mathcal H^{\cont}$, but we will refer to the discrete case when relevant.

\paragraph{Bilevel form.}
The classical problem~\eqref{eq:cep} can be rewritten as
\begin{equation} \label{eq:cep-bilevel}
\begin{array}{ll}
    \textrm{minimize} \quad & \gamma^T \eta + c(x^*(\eta)) \\[0.25em]
    \textrm{subject to}
    & \eta \in \mathcal H,
\end{array}
\end{equation}
where the variable is now just $\eta \in \reals^K$.
We refer to~\eqref{eq:cep-bilevel} as the \textit{bilevel form} of the classical expansion problem (and refer to~\eqref{eq:cep} as the \textit{joint form}), since it involves optimizing a function of the outcome of another optimization problem~\cite{Pozo2017-kz}.
The bilevel form of classical expansion planning is useful for deriving decomposition-based algorithms that consist of generating cuts on $c^{\textrm{opt}}(\eta) = c(x^*(\eta))$~\cite{Pereira1985-zo}.

\subsection{Multi-value expansion planning}

\textit{Multi-value expansion planning} is the problem of choosing long-term investments to minimize some arbitrary function of the dispatch outcome,
\begin{equation} \label{eq:mep}
\problemtag
\begin{array}{ll}
    \textrm{minimize} \quad & \gamma^T \eta + h(x^*(\eta), \lambda^*(\eta)) \\[0.25em]
    \textrm{subject to}
    & \eta \in \mathcal H,
\end{array}
\end{equation}
where the variable is $\eta \in \reals^K$.
The function $h : \reals^N \times \reals^M \rightarrow \reals$ is the \textit{planner's objective} and may differ from the dispatch cost $c$.
Multi-value expansion planning generalizes classical expansion planning, since we can set $h(x, \lambda)  = c(x)$.
When $h(x, \lambda) \neq c(x)$, the planner chooses $\eta$ to minimize some function of $x^*(\eta)$ and $\lambda^*(\eta)$, given that the dispatch outcomes $x^*(\eta)$ and $\lambda^*(\eta)$ are still chosen to minimize cost.
We refer to~\eqref{eq:mep} as the \textit{implicit form} of multi-value expansion planning because the objective includes an implicit function.
We also refer to the entity that solves multi-value expansion planning as the \textit{planner}, which, in general, need not be the same entity as the operator. 
Multi-value expansion planning can be considered as a general class of bilevel expansion planning problems.

The name \textit{multi-value} comes from the fact that the planner may consider multiple values related to the dispatch outcome, not just cost.
In a modern wholesale electricity market, the operator chooses the dispatch outcome $x^*(\eta)$ and prices $\lambda^*(\eta)$ to minimize total cost, since it coincides with the welfare maximizing and market clearing solution~\cite{Gribik2007-bi}.
However, the planner may not be interested in minimizing total cost;
if the planner, for instance, is a utility purchasing power on behalf of their consumers, then they are interested in minimizing the price they pay for power, i.e., the surplus they capture as a consumer, not the overall cost.
Similarly, other agents investing in network upgrades will have other differing and possibly competing objectives, e.g., a state policy maker funding grid investment might wish to minimize system emissions.
These agents cannot change the cost-minimizing nature of the dispatch model, but they can strategically choose network investments that account for it.
In this way, the implicit form of multi-value expansion planning can be interpreted as a direct formulation of the underlying Stackelberg game: the planner is the leader and chooses $\eta$ in the first round, and the operator is the follower and chooses $x^*$ and $\lambda^*$ in the second round following a deterministic strategy known to both agents.

\subsection{Stochastic multi-value expansion planning}
\label{sec:problem-stochastic}

In many settings, there is uncertainty about the state of the system, e.g., future load, technology costs, and renewable production, after the investment decisions have been made.
We can explicitly model this uncertainty by solving a stochastic version of the multi-value expansion planning problem in which the dispatch model is selected uniformly at random from one of $S$ scenarios.
The dispatch model in scenario $s$ is
\begin{equation} \label{eq:dispatch-stochastic}
\begin{array}{ll}
    \textrm{minimize} \quad & c_s(x) \\[0.25em]
    \textrm{subject to}
    & A_s(\eta) x \leq b_s(\eta),
\end{array}
\end{equation}
where the variable is $x \in \reals^N$, and $c_s : \reals^N \rightarrow \reals$, $A_s : \reals^K \rightarrow \reals^{M \times N}$, and $b : \reals^K \rightarrow \reals^M$ define the costs and constraints of the scenario.
Let $x_s^*(\eta) : \reals^K \rightarrow \reals^N$ and $\lambda_s^*(\eta) : \reals^K \rightarrow \reals^M$ denote the solution to the $s$th dispatch problem and its dual, respectively.
Then \textit{stochastic multi-value expansion planning} is the problem of choosing long-term investments to minimize the expected value of some function of the dispatch outcome,
\begin{equation} \label{eq:mep-stochastic}
\begin{array}{ll}
    \textrm{minimize} \quad & \gamma^T \eta + (1/S) \sum_{s=1}^S h(x_s^*(\eta), \lambda_s^*(\eta)) \\[0.25em]
    \textrm{subject to}
    & \eta \in \mathcal H,
\end{array}
\end{equation}
where the variable is $\eta \in \reals^K$.

\section{Illustrative examples}
\label{sec:example}

\newcommand{\desired}{\textrm{desired}}

In this section we give three examples that illustrate multi-value expansion planning problems on a small 6-bus test network.
We first introduce a simple dispatch model in which generators with linear costs dispatch power across a transmission network using linearized power flow constraints.
Then we consider optimal transmission investment for three different objectives: cost minimization across the system, emissions minimization, and profit maximization for an individual generator.
Finally, we give an example of a dynamic dispatch model with multiple time periods, for which the same planning problems can be solved.

% model
\paragraph{Dispatch model.}
Consider an electricity network with $n$ nodes.\footnote{%
Although $n$ is used as an index for the primal dispatch variables in the previous section, we use it here as the number of nodes in the network in order to remain consistent with standard power systems notation.
Other symbols, namely $c$ and $K$, are slightly overloaded but in a fashion consistent with Section~\ref{sec:problem}.
}
Each node $i$ has electricity demand $p_i^{\desired} \in \reals_+$, which can be curtailed (if needed) at a large cost $\delta \in \reals_+$.
Node $i$ may also have a generator with capacity $g_i^{\max} \in \reals^+$ and linear cost $c_i \in \reals_+$;
nodes without generators have $g_i^{\max} = 0$.
The nodes are connected by a set of $K$ transmission lines.
Each line $k$ has electrical susceptance $\beta_k \in \reals_+$ and per-susceptance capacity $f_k^{\max} \in \reals^+$.
Each line also has a source node and a sink node $\ell_1(k) \neq \ell_2(k) \in \{1, \ldots, n\}$, respectively.
The topology of the whole network can be specified by the incidence matrix $L \in \reals^{n \times K}$ with entries
\begin{equation*}
    L_{ik} = \begin{cases}
        1 & i = \ell_1(k) \\
        -1 & i = \ell_2(k) \\
        0 & \textrm{otherwise}
    \end{cases}.
\end{equation*}
The transmission line flows $f_k \in \reals$ must satisfy both capacity constraints, $-\eta_k f_k^{\max} \leq f_k \leq \eta_k f_k^{\max}$, and voltage phase angle constraints, $f = \diag(\eta) L^T \theta$, where $\theta \in \reals^n$ are the nodal voltage phase angles.
The phase angles are relative to an arbitrary reference node, so we set $\theta_1 = 0$.
The full dispatch model is
\begin{equation} \label{eq:example-dispatch}
\begin{array}{ll}
    \textrm{minimize} \quad & 
    \delta \ones^T p 
    + c^T g
    + \epsilon \left( \| p \|_2^2 + \|g\|_2 + \| f \|_2^2 + \| \theta \|_2^2 \right) \\[0.25em]
    \textrm{subject to}
    & f = \diag(\beta) L^T \theta \\[0.25em]
    & g - p = M f \\[0.25em]
    & \theta_1 = 0 \\[0.25em]
    & 0 \leq p \leq p^{\desired} \\[0.25em]
    & 0 \leq g \leq g^{\max} \\[0.25em]
    & -\diag(\beta) f_j^{\max} \leq f \leq \diag(\beta) f^{\max}
\end{array}
\end{equation}
where the variables are $p, g, \theta \in \reals^n$ and $f \in \reals^K$.
Above, $\epsilon > 0$ is a small constant coefficient of a quadratic term used to ensure the objective is strongly convex.
We consider transmission expansion planning problems in which the susceptances can be modified to meet some objective.
Each line susceptance has a box constraint $\beta_k^{\min} \leq \beta_k \leq \beta_k^{\max}$ and investment cost $\gamma_k \in \reals_+$.
Here, the operation variable is $x = (p, g, \theta, f)$ and the planning variable is $\eta = \beta$.
The dispatch cost function is $c(x) = \delta \ones^T p + c^T g + \epsilon \left( \| p \|_2^2 + \|g\|_2 + \| f \|_2^2 + \| \theta \|_2^2 \right)$, and the constraint functions $A(\eta)$ and $b(\eta)$ include all the linear and constant coefficients in the constraints, respectively;
each coefficient is an affine function of $\eta$.

To make things concrete, we consider a classical 6-node test network from~\cite{Garver1970-vi, Romero1994-si} modified to include generator fuel costs and emissions rates; see Figure~\ref{fig:sixbus-nets}.
The network has 6 nodes and 3 generators;
initially, there 6 transmission lines and one of the generators is disconnected, so not all demand can be satisfied; see Panel~A.
The planner may upgrade any of the existing 6 transmission lines or build new transmission lines between any two nodes in order to optimize their objectives.
However, the generator power outputs and transmission line flows themselves are operational decisions chosen by a system operator to minimize cost and clear the market in~\eqref{eq:example-dispatch}.

We consider the following three examples of planning problems on this network: cost minimization, emissions minimization, and individual-generator profit maximization.

\begin{figure}
    \centering
    \includegraphics[width=6.5in]{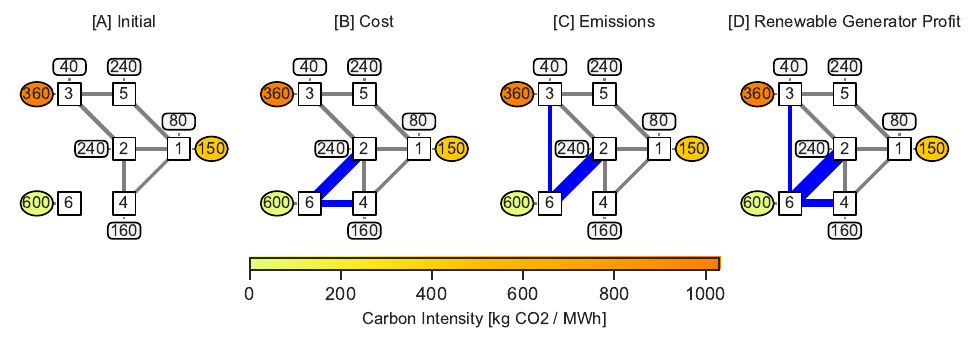}
    \caption{
    Illustrative examples of multi-value expansion planning.
    (A) Initial network before transmission expansion.
    (B) Transmission expansion (blue lines) that minimizes total cost.
    (C) Expansion that minimizes emissions (i.e., emissions weight $w = \infty$).
    (D) Expansion that maximizes the price at node~6.
    }
    \label{fig:sixbus-nets}
\end{figure}

\paragraph{Cost minimization.}
In the cost minimization problem, the planner solves
\begin{equation}
\begin{array}{ll}
    \textrm{minimize} \quad & \gamma^T \eta + c(x^*(\eta)) \\[0.25em]
    \textrm{subject to}
    & \eta \in \mathcal H,
\end{array}
\end{equation}
where the variable is $\eta \in \reals^K$.
This problem corresponds to standard expansion planning and can be reformulated as a joint optimization problem over $p$, $g$, $\theta$, and $\eta$.
However, the resulting problem is nonconvex because the constraint $f = \diag(\beta) L^T \theta$ is bilinear in $\eta$ and $\theta$.
% A number of methods have been developed to address this problem, see~\cite{?}.
The transmission expansion decisions for the 6-bus system are displayed as blue lines in Figure~\ref{fig:sixbus-nets}, Panel~B.

\paragraph{Emissions minimization.}
Suppose each generator has nonnegative emissions rate $e_i \in \reals_+$.
A planner may wish to invest in transmission in order to alleviate congestion that limits the use of low emissions generators.
In the emissions-aware expansion planning problem~\cite{Degleris2021-qg}, the planner solves
\begin{equation}
\label{eq:example-emissions}
\begin{array}{ll}
    \textrm{minimize} \quad & \gamma^T \eta + c(x^*(\eta)) + w e^T g^*(\eta) \\[0.25em]
    \textrm{subject to}
    & \eta \in \mathcal H,
\end{array}
\end{equation}
where the variable is $\eta \in \reals^K$.
(Here, $g^*(\eta)$ represents the optimal choice of $g$ in~\eqref{eq:example-dispatch} given $\eta$.)
The weight $w > 0$ can be interpreted as the implied cost of emissions.
Although emissions are not explicitly incorporated into the generator costs $c$ (e.g., via a carbon tax), the planner has chosen a dollar value to associate with mitigated emissions.
Emissions-aware planning is especially useful for policy makers seeking to reduce emissions in the power grid, and provides a  unified approach for jointly planning economic and policy-driven expansions.

We display results for an emissions-aware expansion planning example in Figure~\ref{fig:sixbus-nets}, Panel~C.
Notably, the optimal transmission investment is different from that of cost minimization in Panel~B: instead of connecting to the loads at nodes~2 and~4, the islanded generator is connected to a high emissions generator at node~3.
To better understand the trade-off between cost and emissions, we plot their approximate Pareto frontier in Figure~\ref{fig:pareto}.
The results suggest strategically selecting transmission investments can lead to a large decrease in emissions at a slight (10-30\%) increase in total cost.

\begin{figure}
    \centering
    \includegraphics[width=3.5in]{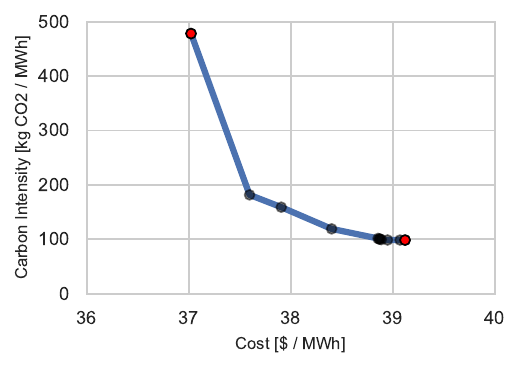}
    \caption{
        Trade-off between total cost (investment cost plus dispatch cost) and emissions for Problem~\eqref{eq:example-emissions} using the 6-node network in Figure~\ref{fig:sixbus-nets}.
        (Black dots) Solutions generated by varying  the emissions weight $w \in [0, 10]$ and solving the resulting multi-value expansion  planning problem using Algorithm~\ref{alg:gradient-descent}.
        (Green line) Pareto frontier for the cost-emissions trade-off.
        (Red dots) The cost-emissions profile of the networks in Figure~\ref{fig:sixbus-nets}, Panels~B and~C.
    }
    \label{fig:pareto}
\end{figure}

\paragraph{Profit maximization.}
If a merchant generator is responsible for funding transmission investment, they are naturally more interested in their own profit than the overall cost of the system.
Let $\nu^*(\eta) \in \reals^n$ be the optimal dual variable of the constraint $g - d = Mf$ in~\eqref{eq:example-dispatch}.
The variable $\nu_i^*(\eta)$ gives the \textit{locational marginal price} at node $i$;
after the operator solves the dispatch problem, they publish the prices $\nu^*(\eta)$ to correctly incentivize each generator to produce the correct amount of power.
The profit maximization problem is
\begin{equation}
\begin{array}{ll}
    \textrm{minimize} \quad & \gamma^T \eta + \sum_{i \in \mathcal G} \nu_i^*(\eta) \cdot g_i^*(\eta)
    \\[0.25em]
    \textrm{subject to}
    & \eta \in \mathcal H,
\end{array}
\end{equation}
where the variable is $\eta \in \reals^K$, and $\mathcal G \subset \{1, \ldots, n\}$ is the set of generators owned by the merchant.\footnote{
    The outer objective $h$ in the profit maximization problem is bilinear and hence nonconvex.
    As we will see later, this means the relaxation derived in Section~\ref{sec:relax} will not be convex either.
    Fortunately, because the objective is bilinear, the same techniques can be applied to relax the terms in the objective as well.
}
% The outer objective in the profit maximization problem is a nonconvex function of $\nu^*(\eta)$ and $g^*(\eta)$ since it includes a bilinear term.
% However, the methods that will be presented in Section~\ref{?} can still be used to solve profit maximization problems or, alternatively, we can replace the objective with its linear approximation at some initial point.
In the 6-bus example, we consider a profit maximization problem for a merchant that owns generator $\mathcal G = \{ 6 \}$ and is paying for a transmission upgrade to connect to the system.
We show the resulting expansion decisions in Figure~\ref{fig:sixbus-nets}, Panel~D.
Note that profit maximization leads to a different set of network upgrades than either cost (Panel~B) or emissions minimization (Panel~C); the additional transmission allows for increased export capacity from node six and therefore higher local prices.

\subsection{Dynamic model}

We can extend the dispatch model in~\eqref{eq:example-dispatch} to be dynamic and model multiple time periods.
The model is exactly the same as~\eqref{eq:example-dispatch}, except that it is solved over $T$ time periods with battery storage devices.
It is given by the optimization problem,
\begin{equation} \label{eq:dynamic-dispatch}
\begin{array}{ll}
    \textrm{minimize} \quad & 
    \sum_{t=1}^T \left(
    \delta \ones^T p_t 
    + c^T g_t
    + \epsilon \left( \|p_t\|_2^2 + \|g_t\|_2 + \|f_t\|_2^2 + \|\theta_t\|_2^2 + \|s_t\|_2^2 + \|v_t\|_2^2 \right) 
    \right) \\[0.25em]
    \textrm{subject to}
    & f_t = \diag(\beta) L^T \theta_t \\[0.25em]
    & g_t - p_t - v_t = M f_t \\[0.25em]
    & s_t = s_{t-1} + v_t \\[0.25em]
    & \theta_1 = 0 \\[0.25em]
    & 0 \leq p_t \leq p_t^{\desired} \\[0.25em]
    & 0 \leq g_t \leq g_t^{\max} \\[0.25em]
    & 0 \leq s_t \leq \diag(\omega) v^{\max} \\[0.25em]
    & -v^{\max} \leq v \leq v^{\max} \\[0.25em]
    & -\diag(\beta) f_j^{\max} \leq f_t \leq \diag(\beta) f^{\max}
\end{array}
\end{equation}
where the variables are $p_t, g_t, \theta_t, s_t, v_t \in \reals^n$ and $f_t \in \reals^K$ for $t = 1, \ldots, T$.
Note that the desired load at node $i$ is now the time varying quantity $(p_t^{\desired})_i \in \reals_+$.
Similarly, the maximum nodal generation capacity $(g_t^{\max})_i$ also varies with time.
Costs and transmission capacities are considered fixed across time, but could be made time varying if desired.
The model also includes lossless battery storage at each node with maximum (dis)charge power $v^{\max}_i \in \reals_+$ and duration $\omega_i \in \reals_+$. 
The battery dynamics are given by the constraint $s_t = s_{t-1} + v_t$, with $s_0 = 0$ as an initial condition.
In this problem, we set the planning variable to be $\eta = (\beta, g^{\max}, v^{\max})$, jointly planning all three device capacities.

The main difference between the static and dynamic models is the addition of storage (with expandable capacities).
Planning problems for the dynamic model are almost identical to their static counterparts.
For example, the emissions minimization problem is
\begin{equation}
\label{eq:example-emissions-dynamic}
\begin{array}{ll}
    \textrm{minimize} \quad & \gamma^T \eta + c(x^*(\eta)) + w \sum_t e^T g_t^*(\eta) \\[0.25em]
    \textrm{subject to}
    & \eta \in \mathcal H,
\end{array}
\end{equation}
where the variable is $\eta \in \mathcal H$.
The dynamic model is, of course, straightforward to extend in the usual ways.
For example, we could add other devices like DC transmission lines or flexible loads.
We could also add additional constraints, such as ramping constraints or lossy batteries.
One benefit of our general planning framework is that it is simple to extend all our results to these cases by changing $c$, $A$, $b$, and the resulting Jacobian $\partial z^*(\eta)$.

\section{Scalable gradient algorithm}
\label{sec:gradient-algorithm}

% Although branch-and-bound algorithms produce an exact solution, their exponential computational complexity makes impractical for solving multi-value expansion planning problems on real-world electricity systems.
% Similarly, directly applying heuristic methods to the nonconvex quadratic program~\eqref{eq:mep-sd} is often slow and leads to poor quality solutions because of the difficulty of satisfying various bilinear constraints.

In this section, we propose a local gradient algorithm that can be applied directly to the implicit form of multi-value expansion planning in~\eqref{eq:mep}.
We first explain how to differentiate the solution maps $x^*(\eta)$ and $\lambda^*(\eta)$ using the \textit{implicit function theorem}.
Then, we derive a simple gradient-based algorithm that converges to a local optimum in polynomial time.
We generalize this algorithm to the stochastic setting and give two heuristic improvements that empirically seem to improve the convergence rate and the solution quality.

\subsection{Implicit differentiation}

% \paragraph{KKT conditions.}
The vectors $x \in \reals^N$ and $\lambda \in \reals^M$ are solutions to the primal and dual dispatch models~\eqref{eq:dispatch} and~\eqref{eq:dual-dispatch}, respectively, if and only if they satisfy $\kappa(x, \lambda, \eta) = 0$, where
\begin{equation}
\label{eq:kkt}
    \kappa(x, \lambda, \eta) = \begin{bmatrix}
        \nabla c(x) + A(\eta)^T \lambda \\
        \diag(\lambda) ( A(\eta) x - b(\eta) )
    \end{bmatrix}.
\end{equation}
The operator $\kappa : \reals^{N+M+K} \rightarrow \reals^{N+M}$ is the \textit{Karush-Kuhn-Tucker (KKT) operator};
finding a solution to a convex optimization problem and its dual is equivalent to solving the system of equations defined by the problem's KKT operator.
In this way, we can think of the primal and dual dispatch maps $x^*(\eta)$ and $\lambda^*(\eta)$ as the solutions to the nonlinear system of equations $\kappa(x, \lambda, \eta) = 0$ parameterized by $\eta$.
This gives a method for differentiating the dispatch maps.

\begin{theorem}[Implicit Function Theorem,~\cite{Dontchev2014-ck}, Chapter~1B]
\label{thm:ift}
    Consider a continuously differentiable function $\kappa : \reals^D \times \reals^K \rightarrow \reals^D$.
    Fix $\eta_0 \in \reals^K$ and suppose there exists a unique point $z_0 \in \reals^D$ such that $\kappa(z_0, \eta_0) = 0$.
    Then there is a function $z^*(\eta) : \Omega \rightarrow \reals^D$ defined on an open set $\Omega \subset \reals^K$ containing $z_0$ such that $\kappa(z^*(\eta), \eta) = 0$ for all $\eta \in \Omega$.
    Moreover, if $\partial_1 \kappa(z^*(\eta), \eta)$ is invertible, then $z^*(\eta)$ is differentiable with Jacobian,
    \begin{equation*}
        \partial z^*(\eta) = -\partial_1 \kappa(z^*(\eta), \eta)^{-1} \cdot \partial_2 \kappa(z^*(\eta), \eta).
    \end{equation*}
\end{theorem}
In the context of the dispatch model, $z^*(\eta) = (x^*(\eta), \lambda^*(\eta))$.
From here forth, we use $z$ instead of $(x, \lambda)$ whenever possible.
The differentiability and uniqueness assumptions follow directly from Assumptions~\ref{assume:feasible},~\ref{assume:convex-diff}, and~\ref{assume:affine-param}.
The gradient of the multi-value expansion planning objective $J(\eta) = \gamma^T \eta + h(z^*(\eta))$ is then
\begin{equation*}
\begin{aligned}
    \nabla J(\eta)
    &= \gamma + \partial z^*(\eta)^T \cdot \nabla h(\eta).
\end{aligned}
\end{equation*}

\paragraph{Fast computation.}
The Jacobians of the solution maps are only used as part of a matrix-vector product;
therefore, we need not compute $\partial z^*$ explicitly and, hence, need not invert $\partial_1 \kappa$ directly.
Instead, we only need to solve the $(M+N)$-dimensional linear system $(\partial_1 \kappa)^T y  = \nabla h$, which can be accomplished efficiently either using a sparse direct factorization or an indirect method~\cite{Trefethen1997-fv}.

\subsection{Gradient descent}

Assume the investment constraint set $\mathcal H$ is the continuous box $\mathcal H^{\cont}$.
Theorem~\ref{thm:ift} implies that solving Problem~\eqref{eq:mep} amounts to minimizing a differentiable function subject to box constraints.
A standard tool for solving this type of problem is \textit{(projected) gradient descent}, which iteratively updates the network investments using gradient steps projected onto $\mathcal H$.
Evaluating the gradient $\nabla J(\eta)$ requires solving the dispatch model with parameters $\eta$, computing Jacobians of the KKT operator, and solving a linear system.
We describe the full approach, called \textit{implicit (projected) gradient descent} in Algorithm~\ref{alg:gradient-descent}.

Despite the fact that $h(z^*(\eta))$ is defined implicitly through the solution to an optimization problem, gradient descent still provably converges to a local optimum in polynomial time.

\begin{figure}
\centering
\begin{minipage}{.8\linewidth}
\begin{algorithm}[H]
\caption{\textit{Projected gradient descent}}
\label{alg:gradient-descent}
\vspace{1em}
\begin{tabbing}
\textbf{given:} an initial point $\eta^{(0)} \in \mathcal H$ and step size $\alpha > 0$ \\
\textbf{while:} stopping criterion is not satisfied \\
\qquad  \= 1. solve the dispatch model~\eqref{eq:dispatch} and its dual~\eqref{eq:dual-dispatch} with parameter $\eta^\iter$ \\
        \> 2. compute the gradient $\nabla J(\eta^\iter) = \gamma + \partial z^*(\eta^\iter)^T \cdot \nabla h(z^*(\eta^\iter))$ \\
        \> 3. set $\eta^{\iterplus} := \proj_{\mathcal H}\left( \eta^\iter - \alpha \nabla J(\eta^\iter) \right)$
\end{tabbing}
\vspace{1em}
\end{algorithm}
\end{minipage}
\end{figure}

\begin{theorem}[Gradient descent convergence]
\label{thm:gradient-descent}
    Assume $h$ is differentiable with Lipschitz continuous gradient.
    Then, for sufficiently small step size, projected gradient descent (Algorithm~\ref{alg:gradient-descent}) finds an $\epsilon$-stationary point,
    \begin{equation*}
        \| \eta^\iter - \proj_{\mathcal H}\left( \eta^\iter - \alpha \nabla J(\eta^\iter) \right) \|_2^2 \leq \epsilon,
    \end{equation*}
    in $O(1/\epsilon^2)$ iterations, where the complexity is independent of the network parameter dimension $K$.
\end{theorem}

\begin{proof}
    See Appendix~\ref{appendix:gradient-descent}.
\end{proof}

Since the dispatch model~\eqref{eq:dispatch} is convex, the dispatch map $z^*(\eta)$ and its gradient $\partial z^*(\eta)$ can be evaluated in time polynomial in the problem dimension using, for example, an interior point method~\cite{Boyd2004-tx}.
Therefore, each iteration of Algorithm~\ref{alg:gradient-descent} has polynomial complexity, and, hence, the algorithm reaches an $\epsilon$-stationary point in polynomial time.

Implicit gradient descent has been used for bilevel optimization in a number of settings~\cite{Pedregosa2016-rr, Lorraine2019-ja, Donti2021-ag}.
However, as far as we are aware, this method has yet to be applied to power system expansion planning problems, save a recent work applying gradient methods to emissions-aware problems~\cite{Degleris2021-qg}.
The closest related methods are \textit{sensitivity methods}, e.g.,~\cite{Pereira1985-sn}, which use gradients in a heuristic manner to solve classical grid expansion planning problems.
These methods, however, use duality to explicitly derive the gradient of operation cost $c(x^*(\eta))$, and therefore cannot be generally applied to generic bilevel objectives.

\paragraph{Interpretation as sensitivity-based investment.}
In some cases, gradient descent has a natural interpretation as a sensitivity-based investment process.
In this process, planners observe the current state of the network parameters $\eta^\iter$ and its dispatch outcome $z^*(\eta^\iter)$.
They also observe $\partial z^*(\eta^\iter)$ the sensitivities of the various grid outcomes with respect to said network parameters.
For example, the sensitivity of operation cost $c(x^*(\eta))$ with respect to additional generation capacity at a node in the network is a function of that node's locational marginal price,
and the sensitivity of emissions is a function of the node's \textit{locational marginal emission rate}~\cite{Ruiz2010-mu, Degleris2021-qg, Valenzuela2023-qp}.
After observing these metrics, the planner makes a decision to invest in slightly more or less of each network parameter.
Then, the new state of the network is revealed and the process repeats.
This process is similar to how the grid is incrementally upgraded and modified in practice in response to changing economic and environmental signals.

\paragraph{Parallelization.}
Algorithm~\ref{alg:gradient-descent} can be applied to stochastic multi-value planning problems~\eqref{eq:mep-stochastic} as well.
The gradient of~\eqref{eq:mep-stochastic} is
\begin{equation*}
    \nabla J^{\textrm{stochastic}}(\eta) = \gamma + \frac{1}{S} \sum_{s=1}^S \partial z_{s}^*(\eta)^T \cdot \nabla h(z_{s}^*(\eta)).
\end{equation*}
The summands $\partial z_{s}^*(\eta)^T \cdot \nabla h(z_{s}^*(\eta))$ can be computed in parallel over $S$ processors.
If a machine with many cores is available, this can drastically reduce the complexity of gradient descent applied to large stochastic problems.

\paragraph{Stochastic gradients.}
We can further accelerate gradient descent for stochastic multi-value planning by sampling a subset of the scenarios $\mathcal B \subset \{ 1, \ldots, S\}$ and estimating the gradient stochastically.
This method scales well to problems with hundreds or even thousands of scenarios and has been extensively studied for solving large-scale nonconvex problems in machine learning~\cite{Bottou2010-as}.
We describe this approach, \textit{stochastic gradient descent}, in Algorithm~\ref{alg:stochastic-gradient-descent}.

\paragraph{Initialization and stopping criterion.}
We use two simple initialization schemes: uniform random sampling from the feasible set $\mathcal H$ and using the solution to the convex relaxation in Section~\ref{sec:relax}.
When initializing using random sampling, we find it is often helpful to sample multiple times---say, five to ten---and choose the sample with the lowest objective value as an initialization.
As a stopping criterion, we terminate the algorithm when either: (i) $\eta^\iter$ is approximately stationary as defined in Theorem~\ref{thm:gradient-descent}, (ii) the objective value is sufficiently close to the lower bound produced by the relaxation in Section~\ref{sec:relax}, or (iii) a fixed iteration limit is reached.

\begin{figure}
\centering
\begin{minipage}{.8\linewidth}
\begin{algorithm}[H]
\caption{\textit{Stochastic projected gradient descent}}
\label{alg:stochastic-gradient-descent}
\vspace{1em}
\begin{tabbing}
\textbf{given:} an initial point $\eta^{(0)} \in \mathcal H$, step size $\alpha > 0$, and a batch size $0 < R < S$ \\
\textbf{while:} stopping criterion is not satisfied \\
\qquad  \= 1. sample $s_1, \ldots, s_R$ uniformly at random from $\{ 1, \ldots, S\}$ \\
        \> 2. solve~\eqref{eq:dispatch-stochastic} and its dual with parameter $\eta^\iter$ for each scenario $s_1, \ldots, s_R$ \\
        \> 3. compute a stochastic estimate of the gradient \\[0.5em]
        \> \qquad\qquad $\Delta^\iter = \gamma + \frac{1}{R} \sum_r \partial z_{s_r}^*(\eta^\iter)^T \cdot \nabla h(z_{s_r}^*(\eta^\iter))$ \\[0.5em]
        \> 4. set $\eta^{\iterplus} := \proj_{\mathcal H}\left( \eta^\iter - \alpha \Delta^\iter \right)$
\end{tabbing}
\vspace{1em}
\end{algorithm}
\end{minipage}
\end{figure}

\paragraph{Integer constraints.}
There is no straightforward extension of gradient descent to the integer-constrained case.
A simple heuristic is to solve the continuous problem, then use probabilistic rounding~\cite{Raghavan1987-on} to generate a few integer solutions.
In many cases, this method works well and finds high quality solutions, especially when the number of decisions is large.

\section{Convex relaxation-based lower bound}
\label{sec:sd-bound}

We now show how to reformulate multi-value expansion planning as a nonconvex quadratic program using strong duality
and derive a simple convex relaxation of this problem using the McCormick envelope.
This convex relaxation gives a lower bound on the optimal objective value of~\eqref{eq:mep} and often yields a good initial network investment for the gradient algorithm.
Finally, we show that the relaxation is exact when the network parameters are binary constrained and approximate when the network parameters are continuous and constrained over a small box.

\subsection{Strong duality reformulation}

To reduce a bilevel optimization problem to an single level problem, we can formulate the optimality conditions of the lower level problem as explicit constraints in the upper level problem.
Specifically, the upper level problem can be written as joint optimization problem over $\eta$, $x$, and $\lambda$, then either the KKT or strong duality conditions of the lower level problem can be used to enforce lower level optimality.
See~\cite{Zare2019-fv, Kleinert2023-eg} for an overview and~\cite{Dempe2002-ef} for a comprehensive treatment.
In this work, we will use the strong duality-based approaches because of their simplicity and favorable performance in~\cite{Zare2019-fv}.

For a given $\eta \in \reals^K$, \textit{weak duality} implies that
\begin{equation} \label{eq:weak-duality}
    c(x) \geq u(\lambda, \eta),
\end{equation}
for any $x \in \reals^N$ and $\lambda \in \reals^M$ such that $A(\eta) x \leq b(\eta)$ and $\lambda \geq 0$.
Since the dispatch model~\eqref{eq:dispatch} and its dual~\eqref{eq:dual-dispatch} are convex and concave, respectively, \textit{strong duality} implies that
\begin{equation*}
    c(x^*(\eta)) = u(\lambda^*(\eta), \eta),
\end{equation*}
and, conversely, that if $c(x) = u(\lambda, \eta)$ for $x \in \reals^N$ and $\lambda \in \reals^M$ such that $A(\eta) x \leq b(\eta)$ and $\lambda \geq 0$, then $x = x^*(\eta)$ and $\lambda = \lambda^*(\eta)$, i.e., $(x, \lambda)$ are optimal if and only if $c(x) = u(\lambda, \eta)$.
Because~\eqref{eq:weak-duality} holds for any $x \in \reals^N$ and $\lambda \in \reals^M$, the above statement implies that $c(x) \leq u(\lambda, \eta)$ is equivalent to $(x, \lambda)$ being optimal.
This allows us to reformulate multi-value expansion planning as a joint optimization problem over $\eta$, $x$, and $\lambda$.
The \textit{$\delta$-strong duality formulation} of multi-value expansion planning is
\begin{equation} \label{eq:mep-sd}
\problemtag
\begin{array}{ll}
    \textrm{minimize} \quad & \gamma^T \eta + h(x, \lambda) \\[0.25em]
    \textrm{subject to}
    & c(x) \leq u(\lambda, \eta) + \delta \\[0.25em]
    & A(\eta) x \leq b(\eta) \\[0.25em]
    & \lambda \geq 0 \\[0.25em]
    & \eta \in \mathcal H,
\end{array}
\end{equation}
where the variables are $x \in \reals^N$, $\lambda \in \reals^M$, and $\eta \in \reals^K$.
Above, $\delta \geq 0$ is a small constant used to relax the strong duality constraint.
When $\delta = 0$, \eqref{eq:mep-sd} is equivalent to the implicit form~\eqref{eq:mep} of multi-value expansion planning; any feasible point $(x, \lambda, \eta)$ of \eqref{eq:mep-sd} is such that $\eta$ is feasible for \eqref{eq:mep}, $x = x^*(\eta)$, and $\lambda = \lambda^*(\eta)$.
Since the objective of both problems is the same, this implies the two problems are equivalent.

In practice, we set $\delta$ to a small positive constant so that the constraint $c(x) \leq u(\lambda, \eta) + \delta$ is strictly feasible; this is done to guarantee that certain technical conditions hold for the relaxed problem in Section~\ref{sec:relax}.
Fortunately, in this case, Problem~\eqref{eq:mep-sd} is still a close approximation of~\eqref{eq:mep}, in the following sense.

\begin{definition*}[$\epsilon$-approximation]
    Consider two optimization problems,
    \begin{equation*}
        \min_{(\eta, y) \in \mathcal C_1}\ f(\eta, y)
        \qquad \textrm{and} \qquad \min_{(\eta, z) \in \mathcal C_2}\ g(\eta, z),
    \end{equation*}
    which we refer to as (A) and (B), respectively, where $\mathcal C_1 \subset \reals^{K + N_1}$ and $\mathcal C_2 \subset \reals^{K + N_2}$.
    We say (B) is an \textit{$\epsilon$-approximation to (A) over $\eta$} if for all $(\eta, z) \in \mathcal C_2$, there exists $y \in \mathcal Y$ such that $(\eta, y) \in \mathcal C_1$ and
    \begin{align*}
        | f(\eta, y) - g(\eta, z) | \leq \epsilon.
    \end{align*}
\end{definition*}

\begin{lemma}
    \label{lemma:delta-strong-duality}
    Suppose $h$ is $L$-Lipschitz continuous.
    For all $\epsilon > 0$, there exists some $\delta > 0$ such that the $\delta$-strong duality formulation~\eqref{eq:mep-sd} is an $\epsilon$-approximation of~\eqref{eq:mep} over $\eta$, i.e., for all $(\eta, x, \lambda)$ feasible for~\eqref{eq:mep-sd},
    \begin{equation*}
        | h(x, \lambda) - h(x^*(\eta), \lambda^*(\eta)) | \leq \epsilon.
    \end{equation*}
\end{lemma}

\begin{proof}
    See Appendix~\ref{appendix:strong-dual}.
\end{proof}

Even when $h$ is convex, Problem~\eqref{eq:mep-sd} is generally not convex because of the bilinear terms $A(\eta) x$ and $u(\lambda, \eta) = -\lambda^T b(\eta) - c^\star(-A(\eta)^T \lambda)$.
If $c$ and $h$ are quadratic (or linear), then~\eqref{eq:mep-sd} is a \textit{nonconvex quadratic problem} and is hard to solve in general.
However, generic off-the-shelf solvers for nonconvex quadratic programs---for example, a nonlinear interior point solver such as Ipopt~\cite{Wachter2006-ig}---can often find good locally optimal solutions in reasonable amounts of time.

\subsection{McCormick relaxation}
\label{sec:relax}

Consider the bilinear function $\phi(\alpha, \beta) = \alpha \beta$, where $\alpha, \beta \in \reals$.
The function $\phi$ is nonconvex, but has an intuitive convex (and concave) envelope when there are box constraints on $\alpha$ and $\beta$.

\begin{definition*}[Convex envelope]
    Consider a function $\phi : \reals^N \rightarrow \reals$ and a convex set $\mathcal S \subset \reals^N$.
    Then the \textit{convex envelope} of $\phi$ is the largest convex under-estimator of $\phi$ over $\mathcal S$,
    \begin{equation*}
        \mathbf{conv}_{\mathcal S}(\phi) = \sup \{ f : f \textrm{ convex and } f(x) \leq \phi(x) \textrm{ for all } x \in \mathcal S \}.
    \end{equation*}
    Similarly, the \textit{concave envelope} of $\phi$ is the smallest concave over-estimator of $\phi$ over $\mathcal S$,
    \begin{equation*}
        \mathbf{conc}_{\mathcal S}(\phi) = \sup \{ f : f \textrm{ concave and } f(x) \geq \phi(x) \textrm{ for all } x \in \mathcal S \}.
    \end{equation*}
\end{definition*}

Let $\mathcal B = \{ \alpha : \alpha^{\min} \leq \alpha \leq \alpha^{\max}\} \times \{\beta :
 \beta^{\min} \leq \beta \leq \beta^{\max} \}$ and suppose $(\alpha, \beta) \in \mathcal B$.
The box constraints imply that the product $(\alpha - \alpha^{\min}) (\beta - \beta^{\min} )$ is always nonnegative, which means that
\begin{equation*}
    \alpha \beta - \alpha \beta^{\min} - \alpha^{\min} \beta + \alpha^{\min} \beta^{\min} \geq 0,
\end{equation*}
or, equivalently,
\begin{equation} \label{eq:envelope-a}
    \phi(\alpha, \beta) \geq \alpha \beta^{\min} + \alpha^{\min} \beta - \alpha^{\min} \beta^{\min}.
\end{equation}
Equation~\eqref{eq:envelope-a} gives an affine lower bound on $\phi(\alpha, \beta)$.
Similarly, the nonnegative product $(\alpha^{\max} - \alpha) (\beta^{\max} - \beta)$ generates the lower bound
\begin{equation} \label{eq:envelope-b}
    \alpha \beta \geq \alpha^{\max} \beta + \alpha \beta^{\max} - \alpha^{\max} \beta^{\max}.
\end{equation}
Equations~\eqref{eq:envelope-a} and~\eqref{eq:envelope-b} define the convex envelope of $\phi$ over the intervals $\alpha^{\min} \leq \alpha \leq \alpha^{\max}$ and $\beta^{\min} \leq \beta \leq \beta^{\max}$.
We can use similar logic to derive the concave envelope of $\phi$ as well.

\begin{lemma}[McCormick envelope,~\cite{Al-Khayyal1983-io}] \label{lemma:envelope}
    The convex and concave envelopes of $\phi(\alpha, \beta) = \alpha \beta$ on $\mathcal B$ are
    \begin{equation*}
    \begin{aligned}
        \mathbf{conv}_{\mathcal B}(\phi) &= \max\left(
            \alpha \beta^{\min} + \alpha^{\min} \beta - \alpha^{\min} \beta^{\min},\
            \alpha^{\max} \beta + \alpha \beta^{\max} - \alpha^{\max} \beta^{\max}
        \right), \\
        \mathbf{conc}_{\mathcal B}(\phi) & = \min\left(
            \alpha^{\min} \beta + \alpha \beta^{\max} - \alpha^{\min} \beta^{\max},\
            \alpha \beta^{\min} + \alpha^{\max} \beta - \alpha^{\max} \beta^{\min}
        \right),
    \end{aligned}
    \end{equation*}
    respectively.
\end{lemma}

\begin{remark}
    For $(\alpha, \beta) \in \mathcal B = [\alpha^{\min}, \alpha^{\max}] \times [\beta^{\min}, \beta^{\max}]$, the bilinear constraint $y = \phi(\alpha, \beta)$ has convex relaxation
    $
        \mathbf{conv}_{\mathcal B}( \alpha \beta ) \leq y \leq \mathbf{conc}_{\mathcal B}( \alpha \beta ),
    $
    which, for short, we write as $y \in \envelope( \alpha \beta )$.
\end{remark}

The McCormick envelope is a well-know relaxation for bilinear terms in optimization problems; see, for example, ~\cite{Borraz-Sanchez2016-up} for an application to gas expansion planning and~\cite{Ploussard2017-ph, Goodarzi2019-ai} for applications classical power system expansion planning.
Here, we use Lemma~\ref{lemma:envelope} to derive a convex relaxation on the strong duality formulation of multi-value expansion planning when the planner's objective $h$ is convex.
We first introduce three new variables $Y_{ij} = A(\eta)_{ij} x_j$, $V_{ji} = -A(\eta)_{ij} \lambda_{i}$, and $\mu_i = \lambda_i b(\eta)_i$ that each model a product of two variables.
We then introduce box constraints on any unbounded variables and, finally, replace each of these products with upper and lower bounds on their convex and concave envelopes, respectively.
The \textit{$\delta$-relaxed formulation} of multi-value expansion planning is
\begin{equation} \label{eq:mep-relax}
\problemtag
\begin{array}{ll}
    \textrm{minimize} \quad & \gamma^T \eta + h(x, \lambda) \\[0.25em]
    \textrm{subject to}
    & c(x) \leq -\mu - c^\star(-V \ones) + \delta \\[0.25em]
    & Y \ones \leq b(\eta) \\[0.25em]
    & x^{\min} \leq x \leq x^{\max} \\[0.25em]
    & 0 \leq \lambda \leq \lambda^{\max} \\[0.25em]
    & Y_{ij} \in \mathbf{env}( A_{ij}(\eta) x_j ),
        \quad i \in \{1, \ldots, m\},\ j \in  \{1, \ldots, n\} \\[0.25em]
    & V_{ji} \in \mathbf{env}( A_{ij}(\eta) \lambda_i ),
        \quad i \in \{1, \ldots, m\},\ j \in  \{1, \ldots, n\} \\[0.25em]
    & \mu_{i} \in \mathbf{env}( \lambda_i b_i(\eta) ),
        \quad i \in \{1, \ldots, m\} \\[0.25em]
    & \eta \in \mathcal H,
\end{array}
\end{equation}
where the variables are $\eta \in \reals^K$, $x \in \reals^N$, $\lambda \in \reals^M$, $Y \in \reals^{M \times N}$, $V \in \reals^{N \times M}$, and $\mu \in \reals^M$.
Here we drop the subscript on the $\mathbf{conv}$ and $\mathbf{conc}$ operators and use the box constraints implied by other constraints of the problem.
Box constraints on $A_{ij}(\eta)$ and $b_i(\eta)$ are implied from the constraint $\eta \in H$ and the linearity of $A$ and $b$.
The box constraints on $x$ and $\lambda$ arise from Assumption~\ref{assume:bounded-domain} and are assumed to either arise from problem structure or, if needed, can be selected as arbitrarily large values.
Problem~\eqref{eq:mep-relax} is convex, has the same objective as Problem~\eqref{eq:mep-sd}, and relaxes its feasible set: any feasible point of~\eqref{eq:mep-sd} is also feasible for~\eqref{eq:mep-relax}.

\paragraph{Lower bound and initialization.}
Since Problem~\eqref{eq:mep-relax} is a convex relaxation of~\eqref{eq:mep}, its optimal objective value is a lower bound on that of the original multi-value problem~\eqref{eq:mep}.
Moreover, its solution $\eta^{\textrm{relax}} \in \mathcal H$ is, in many cases, close to the solution to~\eqref{eq:mep}.
Therefore, we can use $\eta^{\textrm{relax}}$ as an initialization for Algorithm~\eqref{alg:gradient-descent}.
In this context, we can interpret Algorithm~\eqref{alg:gradient-descent} as a local refinement applied to the output of our convex relaxation.

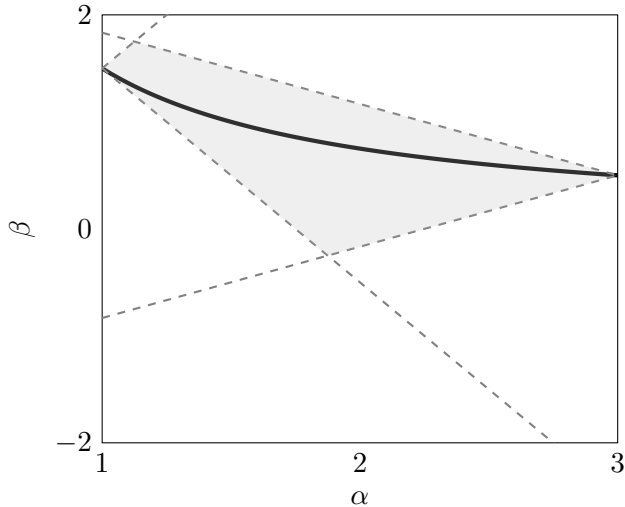
\begin{figure}
\centering
\begin{tikzpicture}

\begin{axis}[
    xlabel={$\alpha$},
    ylabel={$\beta$},
    xmin=1.0, xmax=3,
    ymin=-2, ymax=2,
    grid=none,
    xtick={1, 2, 3},
    ytick={-2, 0, 2},
    tick style={draw=none},
]

% 1.5 = xy
\addplot [domain=1:3, samples=100, smooth, ultra thick, black] {
    1.5/x
};

% Convex envelope
\addplot [domain=1:3, samples=100, smooth, thick, dashed, gray] {
    % x*ymin + xmin*y - xmin*ymin = 1.5
    % x*(-2) + (1)*(y) - (1)*(-2) = 1.5
    % -2*x + y + 2 = 1.5
    -0.5 + 2*x
};
\addplot [domain=1:3, samples=100, smooth, thick, dashed, gray] {
    % x*ymax + xmax*y - xmax*ymax = 1.5
    % x*(2) + (3)*(y) - (3)*(2) = 1.5
    % 2*x + 3y - 6 = 1.5
    (1.5 + 6 - 2*x)/3
};

% Concave envelope
\addplot [domain=1:3, samples=100, smooth, thick, dashed, gray] {
    % xmin*y + x*ymax - xmin*ymax = 1.5
    % (1)*y + x*(2) - (1)*(2) = 1.5
    % y = 1.5 - 2*x + 2
    1.5 - 2*x + 2
};
\addplot [domain=1:3, samples=100, smooth, thick, dashed, gray] {
    % xmax*y + x*ymin - xmax*ymin = 1.5
    % (3)*y + x*(-2) - (3)*(-2) = 1.5
    % 3*y = 1.5 + 2*x - 6
    (1.5 + 2*x - 6)/3
};

% Define the extrema
\coordinate (left) at (axis cs:1, 1.5);
\coordinate (top) at (axis cs:1.125, 1.75);
\coordinate (right) at (axis cs:3, 0.5);
\coordinate (bottom) at (axis cs:1.875, -0.25);

\fill[gray!50, opacity=0.25] (left) -- (top) -- (right) -- (bottom) -- cycle;

% % Fill the triangle with some color
% \fill[green!50,opacity=0.3] (A) -- (B) -- (C) -- cycle;

\end{axis}
\end{tikzpicture}
\caption{
    McCormick envelope of the constraint $\phi = \alpha \beta$ over the box $(\alpha, \beta) \in [1, 3] \times [-2, 2]$. projected onto the plane $\{ (\alpha, \beta, \phi) : \phi = 1.5 \}$.
    (Black line) The nonconvex feasible set.
    (Gray dashed lines) Linear inequalities defined by convex and concave (McCormick) envelopes of $f(\alpha, \beta) = \alpha \beta$ over the boxed region.
    (Gray shaded region) Convexified feasible set.
}
\label{fig:envelope}
\end{figure}

When the network parameters are binary box constrained $\eta \in \mathcal H^{\discrete}$, the relaxation~\eqref{eq:mep-relax} is tight.
Intuitively, this is because the convex envelope of a bilinear function meets the function itself at the endpoints of the interval constraints; see Figure~\ref{fig:envelope}.
As a result, integer-constrained multi-value planning can be converted into a convex mixed-integer program and solved with off-the-shelf methods, if desired.

\begin{theorem}[Tight integer relaxation]
    Consider an integer constrained variable $\alpha \in \{ \alpha^{\min}, \alpha^{\max} \}$ and a box constrained variable $\beta \in [\beta^{\min}, \beta^{\max}]$ (which may or may not have integer constraints as well).
    Then for any $y \in \reals$, the equation $y = \alpha \beta$ holds if and only if
    $
        y \in \envelope( \alpha \beta ).
    $
    As a consequence, if $\mathcal H = \mathcal H^{\discrete}$, then the feasible sets of~\eqref{eq:mep-sd} and its relaxation~\eqref{eq:mep-relax} are equal, i.e., the relaxation is tight.
\end{theorem}

\begin{proof}
    See Appendix~\ref{appendix:integer-envelope}.
\end{proof}

When the network parameters are continuous, \eqref{eq:mep-relax} is not a tight relaxation in general.
However, we can derive an analogous result to the binary case when the width of the box $\mathcal H^{\cont}$ is small: as the width of $H^{\cont}$ goes to zero, the Hausdorff distance between the relaxed set and the original set approaches zero, and so the solution to~\eqref{eq:mep-relax} approaches that of~\eqref{eq:mep-sd}.

\begin{theorem}[Approximate continuous relaxation]
\label{thm:cont-relax}
    % Fix some $\eta^0 \in \mathcal H^{\cont}$.
    For all $\epsilon > 0$, there exists some $\alpha > 0$ such that when
    % \begin{equation*}
    %     \max \left(
    %     \| \eta^{\min} - \eta^0 \|_{\infty},
    %     \| \eta^{\max} - \eta^0 \|_{\infty}
    %     \right) \leq \alpha,
    % \end{equation*}
    $
        \| \eta^{\min} - \eta^{\max} \|_{\infty},
         \leq \alpha,
    $
    then~\eqref{eq:mep-relax} is a $\epsilon$-approximation of~\eqref{eq:mep-sd} over $\eta$.
    % i.e., for all $(\eta, x, \lambda, Y, V, \mu)$ feasible for \eqref{eq:mep-relax}, then there exists some $(x', \lambda')$ such that $(\eta, x', \lambda')$ is feasible for \eqref{eq:mep-sd} and
    % \begin{equation*}
    %     | h(x, \lambda) - h(x', \lambda') | \leq \epsilon.
    % \end{equation*}
\end{theorem}

\begin{proof}
    See Appendix~\ref{appendix:cont-envelope}
\end{proof}

Theorem~\ref{thm:cont-relax} implies that the lower bound from the solution to~\eqref{eq:mep-relax} approaches the optimal value of~\eqref{eq:mep-sd} as $\| \eta^{\min} - \eta^{\max} \|_{\infty} \rightarrow 0$.
As a result, \eqref{eq:mep-relax} can be used as part of a branch-and-bound scheme to solve multi-value expansion planning exactly~\cite{Lenz2022-cz}, albeit in exponential time in the worst-case.

\paragraph{Reformulation and relaxation for stochastic problems.}
Both the strong duality formulation~\eqref{eq:mep-sd} and its relaxation~\eqref{eq:mep-relax} are straightforward to extend to stochastic multi-value expansion planning.
In this case, the strong-duality formulation is a joint optimization problem over the network parameters $\eta$ and the primal variables $x_1, \ldots, x_S$ and dual variables $\lambda_1, \ldots, \lambda_S$ for each scenario.
The resulting optimization problems have $O(SN)$ variables and $O(SM)$ constraints; this may make them prohibitively expensive to solve, especially for the nonconvex strong-duality formulation and for problems with integer-constrained network investments.

\section{Computational experiments}

In this section, we describe our software implementation of the methods developed in Sections~\ref{sec:gradient-algorithm} and~\ref{sec:sd-bound} and analyze their performance on various sized synthetic models of the U.S.\ Western Interconnect in four experiments.
First, we examine the convergence properties of Algorithm~\ref{alg:gradient-descent} and show how it can be used to refine solutions produced by the convex relaxation.
Second, we compare the performance of Algorithm~\ref{alg:gradient-descent} with that of a traditional interior point method applied directly to the nonconvex strong duality formulation of multi-value expansion planning.
Third, we study how the runtime of Algorithm~\ref{alg:gradient-descent} scales with both the dimension of the problem and the number of scenarios.
Finally, we explore the resulting solutions on a planning problem with a long time horizon.

In our experiments, we solve multi-value expansion planning problems using the same dispatch model from Section~\ref{sec:example}.
Unlike Section~\ref{sec:example}, however, we jointly plan \textit{both} transmission and generation capacities, i.e., $\eta = (\beta, g^{\max})$ in the notation of Section~\ref{sec:example}.
In each experiment, we solve the emissions-aware expansion planning problems from~\eqref{eq:example-emissions} with an emissions penalty of \$400-equivalent per metric ton $\textrm{CO}_2$.
In the first three experiments, Section~\ref{sec:exp-converge} through Section~\ref{sec:exp-scaling}, we treat a single hour as one scenario, ignoring dynamic devices like batteries for simplicity, using the dispatch model in~\eqref{eq:example-dispatch}.
In the large example in Section~\ref{sec:exp-large}, we use the dynamic dispatch model from~\eqref{eq:dynamic-dispatch} with 24 hourly periods, treating a full day as a single scenario and jointly planning generation, transmission, \textit{and} storage capacities.

\paragraph{Software implementation.}
We implement the methods developed in Sections~\ref{sec:gradient-algorithm} and~\ref{sec:sd-bound} in Julia~\cite{Bezanson2017-lz}.
We build the dispatch model from Section~\ref{sec:example} using \verb_Convex.jl_~\cite{Udell2014-jq} and solve it using Mosek~\cite{MOSEK-ApS2022-cf} to implement the dispatch map $z^*(\eta)$.
The Jacobian $\partial z^*(\eta)$ is implemented manually using a modular construction that supports a variety of devices in a (DC) linearized power flow model.
To compute the gradient $\nabla h(z)$ for various planning objectives, we use the automatic differentiation library \verb_Zygote.jl_~\cite{Innes2018-he}.
Finally, we model the strong duality formulation of multi-value expansion planning and its relaxation using \verb_JuMP.jl_~\cite{Dunning2017-gm}.
The nonconvex strong duality formulation is solved with Ipopt~\cite{Wachter2006-ig}, and the relaxation can be solved with Mosek.
For stochastic problems from Section~\ref{sec:problem-stochastic}, we parallelize (multithread) the evaluation of both the dispatch map and its Jacobian.
% In theory, this should mean evaluating the gradient for a problem with $S$ scenarios on a machine with $S$ processors takes the same time as for a problem with one scenario.
% In practice, we observe a small overhead with parallelization; see Section~\ref{sec:exp-scaling} and Figure~\ref{fig:scaling} for more details.
All our experiments are run on a single AMD EPYC 7763 processor with 64 cores (128 virtual cores) and 256 GB of RAM at NERSC.

\paragraph{Western Interconnect dataset.}
We use PyPSA-USA~\cite{Tehranchi2023-jo}, an open source model of the U.S.\ Western Interconnect, to load and process network models.
The dataset includes a synthetic network with 4,786 nodes with 8760 hours of data from 2019.
The network granularity can also be adjusted using a cluster algorithm down to as few as 30 nodes, allowing us to create planning cases with dispatch models of different dimensions.
For example, clustering the network to 100 nodes results in a problem with 222 transmission lines and 379 generators, whereas clustering to 500 nodes results in 1108 and 1356 generators, respectively.
In all of our experiments, we consider a future grid scenario with 50\% load growth relative to 2019.
We reduce generation and transmission capacities by 20\% to account for retirements and outages, leaving all other data the same.
The dataset includes realistic fuel costs, investment costs, hypothetical expansion capacities, electrical network parameters, and energy storage devices. 
However, we emphasize that our grid model is not a finely calibrated dataset suitable for long-term policy recommendations, but rather a realistic grid model for testing the methods developed in this paper.
Additional details on the dataset are available in Appendix~\ref{appendix:pypsa}.

\subsection{Convergence analysis}
\label{sec:exp-converge}

In our Experiment~1, we analyze the convergence of gradient descent for varying step sizes.
We test this on two networks, one with 100 nodes and another 500 nodes, both for a deterministic planning problem and a stochastic one with four scenarios.
In each example, we use the convex relaxation from Section~\ref{sec:relax} to generate a lower bound on the optimal objective value and to initialize Algorithm~\ref{alg:gradient-descent}.

In this experiment and Experiment~2, we consider the \textit{suboptimality gap},
\begin{equation} \label{eq:subgap}
    \textrm{gap} = \frac{ J^\iter }{ J^{\mathrm{lower}} } - 1,
\end{equation}
where $J^{\iter}$ is the objective value of the gradient algorithm and $J^{\mathrm{lower}}$ is the lower bound given by the solution to Problem~\eqref{eq:mep-relax}.
The lower bound may not be tight, in which case the suboptimality gap at a global optima will be greater than zero.
However, a suboptimality gap of zero guarantees global optimality, and a suboptimality gap of $\alpha$ guarantees that objective value of the solution found is no worse than $1 + \alpha$ times the objective value at a global optimum.

We test step sizes from $10^{-3}$ to $10^{-1}$ for each case, displaying our results in Figure~\ref{fig:convergence}.
We terminate gradient descent after 2000 iterations, although it is clear that many of the runs converge more quickly.
As expected by Theorem~\ref{thm:gradient-descent}, gradient descent always converges to a stationary point when the step size is sufficiently small.
Larger step sizes may not necessarily converge (see the 100 node, 4 scenario case), but generally converge faster than smaller step sizes.
Regardless, gradient descent with every tested step size improves upon the relaxation-based solution.
Finally, we note that the suboptimality gap~\eqref{eq:subgap} is less than 10\% in the four tested cases.

\begin{figure}
    \centering
    \includegraphics[width=6.5in]{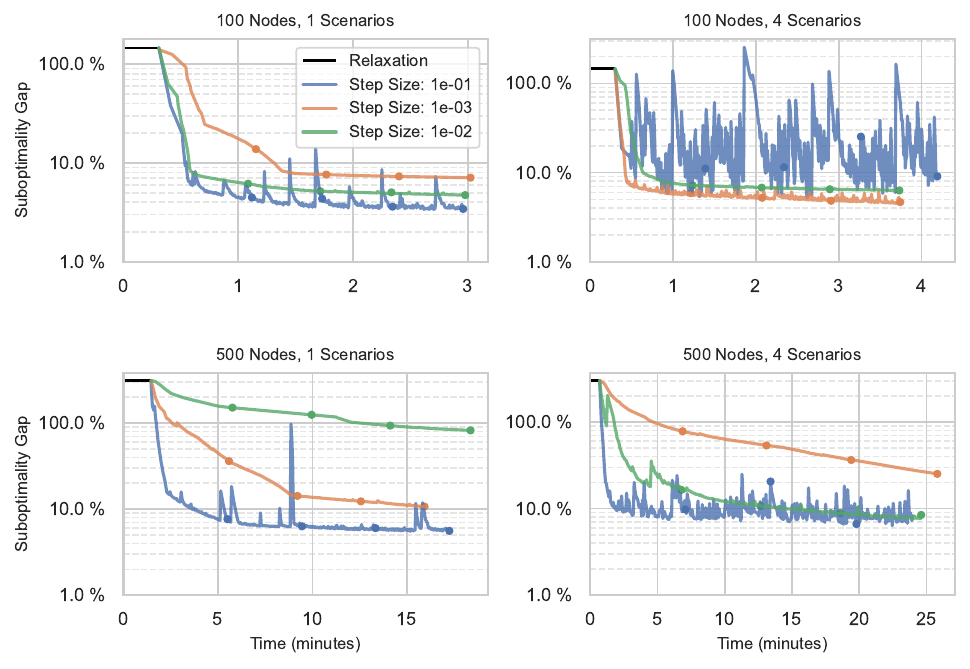}
    \caption{
        Convergence results for gradient descent with varying step sizes.
        (Dashed black line) Lower bound produced by the relaxation from Section~\ref{sec:relax}.
        (Solid black line) Initial choice of the network parameter, produced by the same relaxation.
        The width of the line is the time to solve the relaxation.
        (Solid colored lines) Loss curves for gradient descent.
        Vertical dashed lines mark every 250th iteration.
    }
    \label{fig:convergence}
\end{figure}

\subsection{Comparison to local interior point methods}
\label{sec:exp-comparison}

In Experiment~2, we compare gradient descent to solving the strong duality formulation directly using a local interior point solver.
Specifically, we apply the Ipopt solver to Problem~\eqref{eq:mep-sd}, comparing its objective value and runtime to running 3,000 iterations of gradient descent initialized at the solution to the convex relaxation~\eqref{eq:mep-relax} for problems with 100 nodes and 1, 4, and 16 scenarios.
For each number of scenarios, we sample 10 different sets of hourly data to obtain robust estimates of performance.
Because the interior point solver's runtime can vary significantly from case to case, we limit its runtime to 10 minutes times the number of scenarios.
If Ipopt has not converged, we use the best objective value  solution found so far, with a relatively high feasibility tolerance of $1.0$.

We show the results of this experiment in Figure~\ref{fig:comparison-ipopt}, normalizing objective values using the convex relaxation-based lower bound.
For 1 scenario problems, the interior point algorithm outperforms gradient descent both in objective value and runtime by a small margin.
For 4 scenario problems, the interior point algorithm and gradient descent find similar quality solutions, but gradient descent takes significantly less time.
Finally, for 16 scenario problems, gradient descent significantly outperforms the interior point algorithm both in objective value and runtime.
Note that the interior point algorithm is terminated early for every instance of the 16 scenario problem.
% the results on smaller cases suggest that given sufficient time, the interior point algorithm will likely find a similar quality solution to that of gradient descent.
This highlights the practical importance of developing planning algorithms that scale well to many scenarios.

\begin{figure}[t]
    \centering
    \includegraphics[width=6.5in]{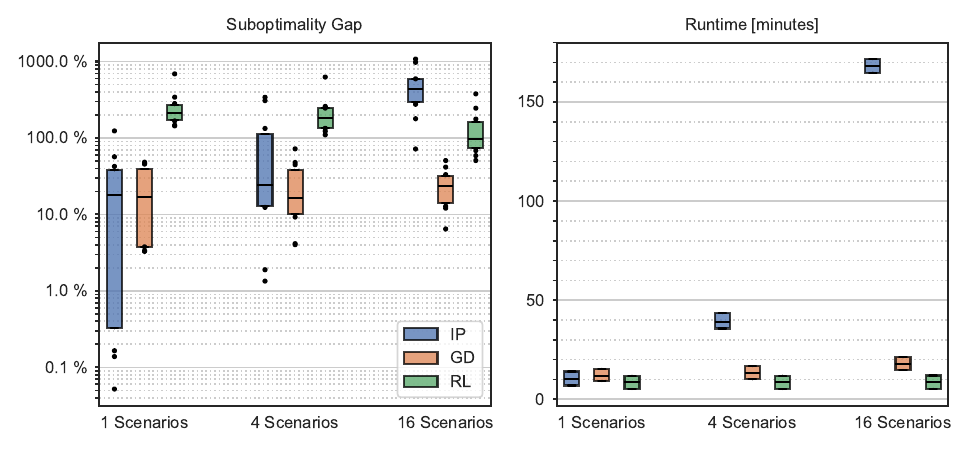}
    \caption{
        Performance comparison between a local interior point solver (IP), gradient descent (GD), and the solution to the relaxation in Section~\ref{sec:relax} (RL).
        (Left side) Box plots of final objective values (divided by the lower bound) for different cases with 1, 4, or 16 scenarios.
        Box bottom and top lines are the 25th and 75th percentiles, respectively, and midlines are medians.
        (Right side) Box plot of algorithm runtimes for different cases.
        For visibility, box bottom and top lines are set to be 10 minutes apart if the IQR is less than 10 minutes.
    }
    \label{fig:comparison-ipopt}
\end{figure}

\subsection{Computational scaling}
\label{sec:exp-scaling}

One of the primary advantages of using gradient descent to solve planning problems is that it scales well to large networks with many scenarios.
In Experiment~3, we verify this by empirically studying the computational time per iteration and, separately, the number of iterations required to converge to a stationary point.

\paragraph{Computational performance.}
First, we study the computational time per iteration for cases with network size varying from 200 to 1000 nodes, solving the resulting planning problems with a single scenario.
Each case is run 3 times for 500 iterations, and the per-iteration time averaged across the duration of the run.
We plot these results in Figure~\ref{fig:scaling} (left panel).
The per-iteration runtime appears to grow approximately linearly with network size.
This roughly agrees with theoretical intuition: the main computational bottleneck is factorizing the sparse KKT matrix while solving the dispatch model, which grows near-linearly with problem dimension for many structured matrices (see, for example,~\cite{Cohen2018-qf}).

Second, we study how the per-iteration computational time scales with number of scenarios.
Since we are able to parallelize the majority of the algorithm\,---\,solving the dispatch model and evaluating gradients for different scenarios\,---\,we expect a near-constant scaling on a machine with many processors.
We use a 100-node network and scale the number of scenarios from 1 to 32, again running each case 3 times.
Experiments are run with 64 threads on a 128 core (64 physical cores) cluster machine.
Therefore, in principle, runtime should be constant as a function of the number of scenarios, barring competing parallelism between linear algebra libraries and the solver.
The results are plotted in Figure~\ref{fig:scaling} (right panel).
For convenience, we plot worst (linear) and best (constant) case theoretical scaling in the same panel.
Results demonstrate that evaluating gradients in parallel is significantly faster than a serial implementation and achieves a near-constant scaling with the number of scenarios.

\begin{figure}
    \centering
    \includegraphics[width=\textwidth]{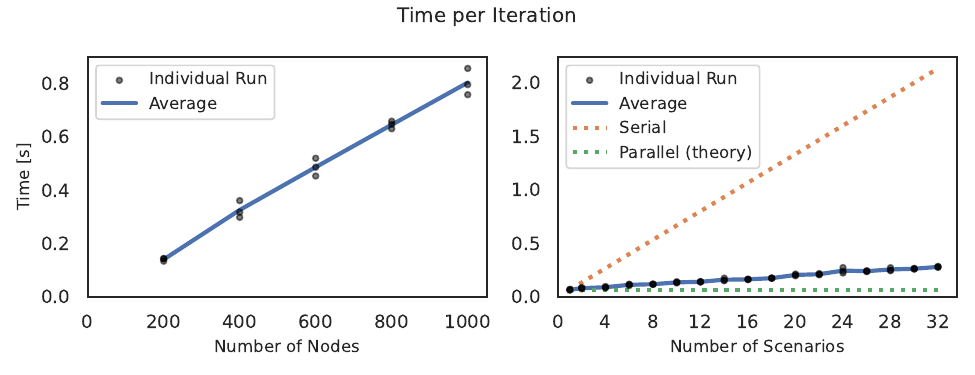}
    \caption{
        Per-iteration computational time for various cases.
        Black dots denote runtimes of individual cases and blue solid lines are averaged runtime.
        (Left) Time per iteration as a function of network size.
        (Right) Time per iteration as a function of number of scenarios.
        Serial performance (orange dotted line), i.e., linear scaling and ideal parallel performance (green dotted line), i.e., constant scaling, are plotted for reference.
        Gradient descent benefits significantly from parallelization.
    }
    \label{fig:scaling}
\end{figure}

\paragraph{Convergence rate.}
Finally, we study the number of iterations until convergence.
Here, we define convergence as when the best objective value so far decreases by less than 0.1\% over 100 iterations.
If the algorithm has not converged after 5,000 iterations, we terminate it early.
We consider a variety of cases as before; we vary the number of nodes from 200 to 1000 on single scenario cases, then vary the number of scenarios from 1 to 16 on 100 node cases.
For each network size, we solve 5 distinct scenarios separately to ensure our results are not an artifact of a particular time period;
for each number of scenarios, we also solve 10 distinct cases for the same reasons.
For each individual case, we select sweep step sizes from $10^{-3}$ to $10^{-1}$ and choose the step size with the best loss at convergence.
We display these results in Figure~\ref{fig:scaling-convergence}.
In general, we observe a slight increase in the number of iterations required for convergence as a function of either network size or number of scenarios.
This may suggest that larger cases have gradients with a larger Lipschitz constant.
In practice, we believe this could be addressed through more sophisticated problem preconditioning or non-constant step size schedules.

\begin{figure}[t]
    \centering
    \includegraphics[width=\textwidth]{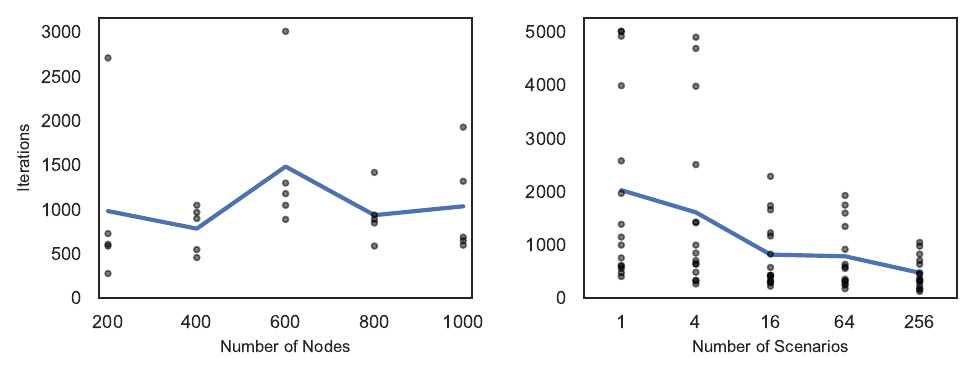}
    \caption{
        Iterations until convergence for different cases on the PyPSA-USA network.
        Black dots denote individual cases and blue solid lines are averages.
        (Left) Iterations until convergence as a function of network size.
        (Right) Iterations until convergence as a function of number of scenarios.
        Gradient descent appears to converge within a constant number of iterations, as suggested by the theoretical analysis in Section~\ref{sec:gradient-algorithm}.
        Faster convergence rates for large cases may be indicative of a shrinking Lipschitz constant.
    }
    \label{fig:scaling-convergence}
\end{figure}

\subsection{Large example}
\label{sec:exp-large}

In Experiment~4, we demonstrate the potential utility of our method by solving a large case study.
We now consider \textit{dynamic} dispatch models with multiple time periods.
Each dispatch problem is solved over $T = 24$ hours coupled by battery storage systems.
We then solve two expansion planning problems using $S = 32$ scenarios (768 hours total) using the 100-node PyPSA-USA network: an emissions-aware problem with a penalty of \$400 per metric ton CO$_2$ (as in the previous experiments) and a cost-only problem without a carbon penalty.
We specifically consider the 16 peak load days and the 16 peak renewable days as our scenarios.
In both problems we jointly plan transmission, generation, and storage.

Because we encounter memory issues when solving the entire problem, we use stochastic gradient descent as described in Algorithm~\ref{alg:stochastic-gradient-descent} with a batch size of $R = 4$.
Instead of sampling batches randomly, we randomly partition the scenarios into batches and iterate through each batch serially, as is standard practice in machine learning.
We set the step size to $0.01$ and run the algorithm for 1000 iterations.
In both cases, the algorithm converges after just a few hundred iterations.
The final objective values are within 2\% and 33\% of the lower bound for the cost-only and emissions-aware cases, respectively.

We plot the resulting transmission, storage, and generation capacity expansions for both cases in Figure~\ref{fig:large-capacity}.
As expected, the emissions-aware problem invests in significantly more renewable, transmission, and storage capacity.
In order to understand each solution qualitatively, we also plot the generation dispatch for the scenario with the highest peak load in Figure~\ref{fig:large-schedules}.
The results show that greater investment in both transmission and storage enable higher renewable penetration, reducing the need for gas capacity.
We reemphasize that these conclusions are not to be interpreted as actual forecasts or policy suggestions for the Western U.S., just as a demonstration of the potential applicability of our methods to large scale case studies.

Quantitatively, the cost of investment is higher in the emissions-aware case: $\$44.3 /\textrm{MWh}$ as opposed to $\$15.9/\textrm{MWh}$ in the cost-only case.
The greater investment in renewable capacity leads to a significant reduction in operation cost ($\$5.0/\textrm{MWh}$ versus $\$16.3/\textrm{MWh}$) and carbon intensity ($47.4\ \textrm{kgCO}_2 / \textrm{MWh}$ versus $186.8\ \textrm{kgCO}_2 / \textrm{MWh}$).
Both cases improve upon the 2019 base case carbon intensity of $348.3\ \textrm{kgCO}_2 / \textrm{MWh}$, but the emissions-aware case achieves an additional $40.0\%$ reduction in carbon intensity at an additional $\$17.1/\textrm{MWh}$ in total cost.

\begin{figure}
    \centering
    \includegraphics[width=6.5in]{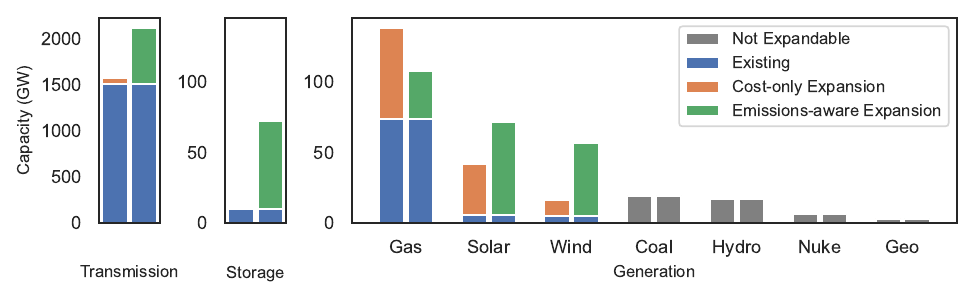}
    \caption{
        Expanded capacity for transmission, storage, and generation using cost-only (orange) and emissions-aware (green) expansion planning models.
        Initial capacities are shown in blue, and non-expandable resources in gray.
        Emissions-aware expansion leads to significantly more transmission, storage, and generation investment, even in a market dispatched solely on cost.
    }
    \label{fig:large-capacity}
\end{figure}

\begin{figure}
    \centering    \includegraphics[width=6.5in]{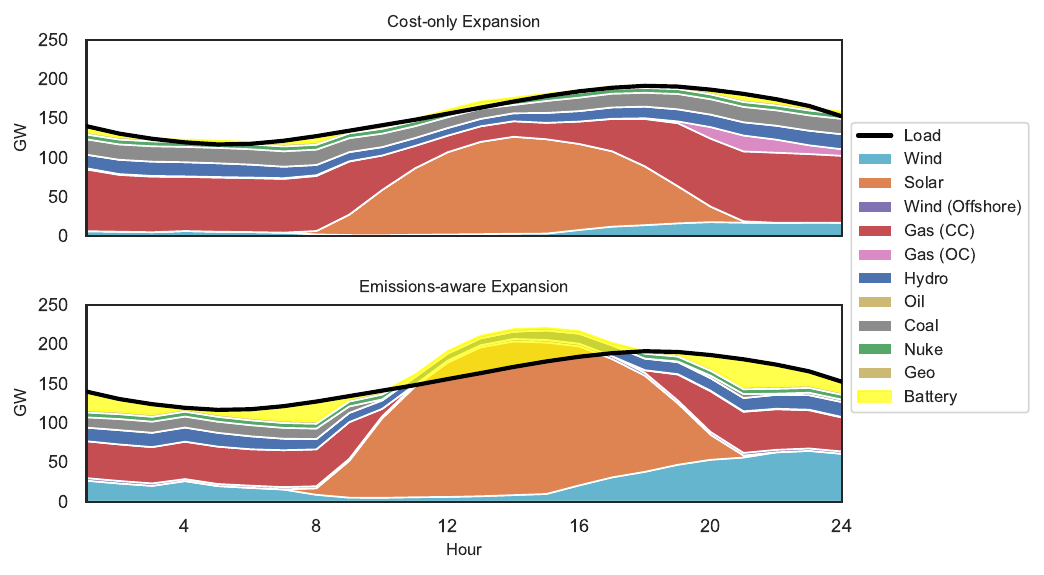}
    \caption{
        Generator outputs on the peak summer load day using new capacities from the cost-only and emissions-aware expansions.
        (Black line) Total hourly load across the network.
        (Shaded solid regions) Generator outputs by resource type.
        (Transparent yellow region) Load shift from battery charging or discharging.
        Cost-only investment reduces carbon intensity by 46.4\% relative to the 2019 baseline, whereas emissions-aware investment leads to a 86.4\% reduction through greater renewable, transmission, and storage investment.
    }
    \label{fig:large-schedules}
\end{figure}

% Initial Solution
% =============
% Average Load: 				130.0 GW
% Average Lost Load: 			20.0 GW
% Investment: 				-0.0064 $/MWh
% Operation: 					30.43 $/MWh
% Carbon: 					348.34 kgco2/MWh
% Full Cost: 					259.77 $/MWh
% =============

% Cost-only Expansion
% =============
% Average Load: 				130.0 GW
% Average Lost Load: 			0.016 GW
% Investment: 				15.865 $/MWh
% Operation: 					16.288 $/MWh
% Carbon: 					186.84 kgco2/MWh
% Full Cost: 					32.215 $/MWh
% =====================================

% Emissions-aware Expansion
% =============
% Average Load: 				130.0 GW
% Average Lost Load: 			7.1e-7 GW
% Investment: 				44.336 $/MWh
% Operation: 					4.9942 $/MWh
% Carbon: 					47.365 kgco2/MWh
% Full Cost: 					68.276 $/MWh
% =====================================

\section{Conclusion}

We consider multi-value expansion planning, a general bilevel optimization framework for jointly optimizing generation, transmission, and storage assets in settings where the planner's objective may differ from the objective used to dispatch power.
We give several examples of multi-value expansion planning, such as reducing emissions, even in the absence of a carbon tax, or maximizing renewable profits.
In general, this problem is high dimensional, nonconvex, and may need to be solved across many scenarios.
Therefore, we propose applying gradient descent to the implicit form of multi-value expansion planning, which converges to a local stationary point in polynomial time.
We demonstrate, both theoretically and empirically, that this method is scalable to high dimensional problems with many scenarios.
Moreover, we show gradient descent outperforms commonly used strong duality-based reformulation methods even on modest sized cases.

One strength of our work is that the multi-value formulation is fairly general and can be extended to a variety of related expansion planning problems.
For example, we can consider planning electricity system expansions jointly with other resources, e.g., gas or water, with no change to our gradient algorithm other than the implementation of $x^*(\eta)$ and $\partial x^*(\eta)$.
We can also consider more advanced variants of expansion planning, such as problems with multiple investment periods.

% We note three important areas for future work.
% First, the solution to the dispatch model changes very little between iterations.
% Warm start methods could potential exploit this to greatly reduce the per-iteration computation time.
% Second, the solution to multi-value expansion planning itself may only change a little when modifying various assumptions and objectives in the problem.
% Therefore, we believe this method could be suitable to exploring spaces of near-optimal solutions~\cite{Neumann2021-hp} and performing sensitivity analyses.
% Finally, the proposed method is designed specifically to solve continuous expansion planning problems.
% Future work exploring modifications for problems with integer constraints is needed.

\section*{Acknowledgements}

This research used resources of the National Energy Research Scientific Computing Center (NERSC), a U.S.\ Department of Energy Office of Science User Facility located at Lawrence Berkeley National Laboratory, operated under Contract No. DE-AC02-05CH11231 using NERSC award DDR-ERCAP0026889.

This material is based upon work supported by the U.S.\ Department of Energy, Office of Science, Office of Advanced Scientific Computing Research, Department of Energy Computational Science Graduate Fellowship under Award Number DE-SC0021110.

\bibliography{main}

\newpage
\appendix

\section{Technical results on gradient descent}

\subsection{Definitions}

\begin{definition*}[Operator norm]
    The \textit{operator norm} of a matrix $A \in \reals^{M \times N}$ is
    \begin{equation*}
        \| A \|_{\op} = \sup_{\| x \|_2 \leq 1}\ \| A x \|_2.
    \end{equation*}
\end{definition*}

\begin{definition*}[Lipschitz continuity]
    A function $f : \Omega \rightarrow \reals^M$ is \textit{$L$-Lipschitz} continuous if for all $x, y \in \Omega$,
    \begin{equation*}
        \| f(x) - f(y) \|_2 \leq L \| x - y \|_2.
    \end{equation*}
    Similarly, a matrix-valued function $A : \Omega \rightarrow \reals^{M \times N}$ is $L$-Lipschitz continuous if for all $x, y \in \Omega$,
    \begin{equation*}
        \| A(x) - A(y) \|_{\op} \leq L \| x - y \|_2.
    \end{equation*}
\end{definition*}

\subsection{Unique dual map}
\label{appendix:unique-dual}

\begin{lemma}
    The solution to the dual dispatch model~\eqref{eq:dual-dispatch} is unique.
\end{lemma}

\begin{proof}
    By Assumption~\ref{assume:convex-diff}, the conjugate function $c^\star$ is both differentiable with $(1/\epsilon)$-Lipschitz continuous gradient and $(1/L)$-strongly convex~\cite{Zhou2018-to}.
    Recall the objective of~\eqref{eq:dual-dispatch} is,
    \begin{equation*}
        u(\lambda) = -\lambda^T b(\eta) - c^\star(-A(\eta)^T \lambda).
    \end{equation*}
    Now consider two optimal solutions $\lambda_1^*, \lambda_2^* \in \reals^M$.
    Since $-c^\star$ is strongly concave, it must be that
    \begin{equation}
        \label{eq:A-lambda}
        A(\eta)^T \lambda_1^* = A(\eta)^T \lambda_2^*.
    \end{equation}
    By complementary slackness, $(\lambda_i)^*_m > 0$ only if the constraint $\tilde a(\eta)^T x^*(\eta) \leq b(\eta)$ is binding.
    Assumption~\ref{assume:non-redundant-constraints} tells us that these rows of $A(\eta)$ are linearly independent, so the sub-matrix of $A(\eta)^T$ corresponding to the non-zero entries of $\lambda_1^*$ and $\lambda_2^*$ has linearly independent columns.
    Therefore,~\eqref{eq:A-lambda} implies that $\lambda_1^* = \lambda_2^*$, i.e., the solution is unique.
\end{proof}

\subsection{Gradient descent convergence}
\label{appendix:gradient-descent}

Theorem~\ref{thm:gradient-descent} is a restatement of the following result.

\begin{theorem}[Gradient descent convergence, full]
\label{thm:gradient-descent-full}
    Let $h$ have Lipschitz continuous gradient.
    Then $J(\eta)$ has Lipschitz continuous gradient with Lipschitz constant $L > 0$ and a lower bound $J^* \leq J(\eta)$ for all $\eta \in \mathcal H$.
    Now suppose Algorithm~\ref{alg:gradient-descent} is run with $\alpha \leq 1 / L$.
    Then there exists an iteration
    \begin{equation*}
        i \leq \left\lceil \frac{2 (J(\eta^1) - J^*)}{L \epsilon^2} \right\rceil
    \end{equation*}
    such that $\eta^i$ is $\epsilon$-stationary,
    \begin{equation*}
        \| \eta^i - \proj_{\mathcal H}\left( \eta^i - \alpha \nabla J(\eta^i) \right) \|_2 \leq \epsilon.
    \end{equation*}
\end{theorem}

At a high level, our approach is to show that the objective function $J(\eta) = \gamma^T \eta + h(z^*(\eta))$ has Lipschitz continuous gradients, then apply a standard analysis for the convergence of projected gradient descent for Lipschitz functions, e.g., as in~\cite{Pedregosa2016-rr} or, for a more comprehensive treatment,~\cite{Poljak1987-wz}.
Since we assume the Lipschitz continuity of $\nabla h(z)$, the main technical challenge is showing that $\partial z^*(\eta)$ is Lipschitz continuous as well (in the metric induced by the operator norm).
Fortunately, this follows from our initial smoothness and regularity assumptions in Section~\ref{sec:dispatch}.

\begin{lemma}
\label{lemma:lipschitz-kkt}
    Let $D = N+M$ and consider the KKT operator $\kappa$ defined in Equation~\ref{eq:kkt}
    Then the matrices $\partial_1 \kappa(z, \eta) \in \reals^{D \times D}$ and $\partial_2 \kappa(z, \eta) \in \reals^{D \times K}$ are Lipschitz continuous.
\end{lemma}

\begin{proof}
    Recall that,
    \begin{equation*}
        \kappa(z, \eta) =
        \begin{bmatrix}
            \nabla c(x) + A(\eta)^T \lambda \\
            \diag(\lambda) ( A(\eta) x - b(\eta) )
        \end{bmatrix},
    \end{equation*}
    where $(x, \lambda) = z$.
    Therefore, the partial Jacobians are,
    \begin{equation*}
        \partial_1 \kappa(z, \eta) =
        \begin{bmatrix}
            \nabla^2 c(x)            & A(\eta)^T \\
            \diag(\lambda) A(\eta)   & \diag(A(\eta) x - b(\eta))
        \end{bmatrix},
        \quad
        \partial_2 \kappa(z, \eta) =
        \begin{bmatrix}
            \sum_n e_n \lambda^T \partial a_n(\eta) \\
            \diag(\lambda) \left(-\partial b(\eta) + \sum_n \partial a_n(\eta) x_n \right)
        \end{bmatrix},
    \end{equation*}
    where $a_n(\eta) \in \reals^M$ is the $n$th column of $A(\eta)$ and $e_n \in \reals^N$ is the $n$th standard unit vector.
    By Assumptions~\ref{assume:convex-diff} and~\ref{assume:affine-param}, each block of the two matrices is Lipschitz continuous.
    It follows that the matrices themselves are Lipschitz continuous (via the triangle inequality).
\end{proof}

\begin{lemma}
\label{lemma:kkt-jac-invertible}
    For any $\eta \in \mathcal H$, the matrix $\partial_1 \kappa(z^*(\eta), \eta)$ is invertible.
\end{lemma}

\begin{proof}
    Throughout the proof, we will drop the argument of functions of $\eta$ for readability.
    Strong convexity implies that the Hessian $H = \nabla^2 c(x)$ is invertible.
    Therefore, by the Schur's formula~\cite{Boyd2004-tx}, $\partial_1 \kappa$ is invertible if and only if
    \begin{equation*}
        S = \diag(A x - b) - \diag(\lambda) A H^{-1} A^T
    \end{equation*}
    is invertible.
    To show this, we will show that $S v = 0$ implies $v = 0$, i.e., $S$ has linearly independent columns.
    So suppose $S v = 0$.
    Denote $\mathcal B \subset \{1, \ldots, M\}$ the set of indices of binding constraints, i.e., the indices $m \in \mathcal B$ such that $\tilde a_m^T x - b_m = 0$.

    First, note that for $m \not\in \mathcal B$, complementary slackness implies that $\lambda_m = 0$.
    Consequently, $( \diag(\lambda) A H^{-1} A^T v )_m = 0$.
    Therefore, since $Sv = 0$, it must be that $(\diag(A x - b) v)_m = (\tilde a_m^T x - b_m) v_m = 0$ as well.
    But since $m \not\in \mathcal B$, we know $\tilde a_m^T x - b_m < 0$, so it must be that $v_m = 0$ instead.
    As a result, we know that $v_m = 0$ for all $m \not\in \mathcal B$.

    Now consider $m \in \mathcal B$.
    By definition, $\tilde a_m^T x - b_m = 0$, so $(\diag(A x - b) v)_m = 0$.
    Therefore, it must be that
    \begin{equation*}
        ( \diag(\lambda) A H^{-1} A^T v )_m
        = \lambda_m (A H^{-1} A^T v)_m
        = 0
    \end{equation*}
    as well.
    By strict complementarity (Assumption~\ref{assume:strict-complement}), $\lambda_m > 0$, so it must be that $(A H^{-1} A^T v)_m = 0$.
    Now denote $A_{\mathcal B} \in \reals^{|\mathcal B| \times N}$ the submatrix of $A$ that includes each of the rows in $\mathcal B$ and similarly denote $v_{\mathcal B} \in \reals^{|\mathcal B|}$ the entries of $v$ in $\mathcal B$.
    Then we can rewrite $(A H^{-1} A^T v)_m = 0$ for $m \in \mathcal B$ as
    \begin{equation*}
        A_{\mathcal B} H^{-1} A_{\mathcal B}^T v_{\mathcal B} = 0.
    \end{equation*}
    Of course, this implies $v_{\mathcal B} A_{\mathcal B} H^{-1} A_{\mathcal B}^T v_{\mathcal B} = 0$ as well.
    Since $H$ is positive-definite, it has  invertible square root $H^{1/2}$, so we can rewrite the above as,
    \begin{equation*}
        \| H^{-1/2} A_{\mathcal B}^T v_{\mathcal B} \|_2^2 = 0,
    \end{equation*}
    which then implies $H^{-1/2} A_{\mathcal B}^T v_{\mathcal B} = 0$.
    Multiplying both sides by $H^{1/2}$, we conclude $A_{\mathcal B}^T v_{\mathcal B} = 0$.
    By Assumption~\ref{assume:non-redundant-constraints}, this can only hold if $v_{\mathcal B} = 0$.

    Since $v_{m} = 0$ for $m \not\in \mathcal B$ and $m \in \mathcal B$, it must be that $v = 0$.
    We conclude that $S$ has linearly independent columns and is therefore invertible.
    As a result, we conclude $\partial_1 \kappa(z^*(\eta), \eta)$ is invertible as well.
\end{proof}

\begin{lemma}
    \label{lemma:non-singular-jac}
    There exists $\epsilon > 0$ such that for all $\eta \in \mathcal H$, the Jacobian $\partial_1 \kappa(z^*(\eta), \eta)$ has minimum singular value
    \begin{equation*}
        \rho(\eta) = \sigma_{\min}(\partial_1 \kappa(z^*(\eta), \eta)) \geq \epsilon.
    \end{equation*}
\end{lemma}

\begin{proof}
We start by noting that Lemma~\ref{lemma:kkt-jac-invertible} implies that the Implicit Function Theorem (Theorem~\ref{thm:ift}) holds for every $\eta \in \mathcal H$.
So $z^*(\eta)$ is differentiable and, hence, continuous, for all $\eta \in \mathcal H$.
As a consequence of this and Lemma~\ref{lemma:lipschitz-kkt}, the Jacobian $\partial_1 \kappa(z^*(\eta), \eta)$ is a continuous function of $\eta$ as well.

Since the minimum singular value is a continuous function (by Weyl's inequality), the function $\rho$ is continuous.
Since $\mathcal H$ is also compact, the Extremal Value Theorem implies that there exists $\hat \eta \in \mathcal H$ such that
\begin{equation*}
    \rho(\hat \eta) = \inf_{\eta \in \mathcal H} \rho(\eta).
\end{equation*}
Of course, $\hat \eta \in \mathcal H$, so $\rho(\hat \eta) = \epsilon > 0$ (otherwise, $\partial_1 \kappa(z^*(\eta), \eta)$ is not invertible).
So $\inf_{\eta \in \mathcal H} \rho(\eta) = \epsilon$, which proves the claim.
\end{proof}

\begin{corollary}
    \label{corollary:cont-bounded-inverse}
    The function $M(\eta) = (\partial_1 \kappa(z^*(\eta), \eta))^{-1}$ is Lipschitz continuous and bounded.
\end{corollary}

\begin{proof}
    Boundedness and, hence, compactness, follows directly from Lemma~\ref{lemma:non-singular-jac}, since
    \begin{align*}
        \left\| (\partial_1 \kappa(z^*(\eta), \eta))^{-1} \right\|
        = \frac{1}{\sigma_{\max}\left( (\partial_1 \kappa(z^*(\eta), \eta))^{-1} \right)}
         = \frac{1}{\sigma_{\min}\left( \partial_1 \kappa(z^*(\eta), \eta) \right)}
        \leq 1 / \epsilon.
    \end{align*}
    Continuity follows from the fact that the inverse map $A \rightarrow A^{-1}$ is continuous~\cite{Stewart1969-br}.
    So $M(\eta)$ is a continuous function on a compact set, and is hence Lipschitz continuous as well.
\end{proof}

\begin{lemma}
    The Jacobian $\partial z^*(\eta)$ and the gradient of the implicit planning objective $\nabla J(\eta)$ are Lipschitz continuous functions over $\eta \in \mathcal H$.
\end{lemma}

\begin{proof}
    Both functions are compositions, sums, and products of bounded, Lipschitz continuous functions on compact sets.
    Therefore, they are Lipschitz continuous.

    For clarity, we will write this out explicitly for the Jacobian of the solution map $\partial z^*(\eta)$.
    Choose $\eta_1, \eta_2 \in \mathcal H$.
    For ease of notation, let $\partial_i \kappa_j$ denote $\partial_i \kappa(z^*(\eta_j), \eta_j)$ for $i, j \in \{1, 2\}$.
    Then
    \begin{align*}
        \| \partial z^*(\eta_1) - \partial z^*(\eta_2) \|
        &= \left\|
            \left( \partial_1 \kappa_1 \right)^{-1} \partial_2 \kappa_1 - \left( \partial_1 \kappa_2 \right)^{-1} \partial_2 \kappa_2
        \right\| \\
        &\leq \left\| \left(
            \left( \partial_1 \kappa_1 \right)^{-1} -
            \left( \partial_1 \kappa_2 \right)^{-1}
            \right) \partial_2 \kappa_1
        \right\| +
        \left\| \left( \partial_1 \kappa_2 \right)^{-1}
            \left( \partial_2 \kappa_1 - \partial_2 \kappa_2 \right)
        \right\| \\
        &\leq \left\|
            \left( \partial_1 \kappa_1 \right)^{-1} -
            \left( \partial_1 \kappa_2 \right)^{-1}
            \right\| \left\| \partial_2 \kappa_1 \right\|
         +
        \left\| \left( \partial_1 \kappa_2 \right)^{-1} \right\|
        \left\| \partial_2 \kappa_1 - \partial_2 \kappa_2 \right\| .
    \end{align*}
    Of course, both $\partial_2 K$ and $(\partial_1 K)^{-1}$  are both Lipschitz continuous and bounded by Lemma~\ref{lemma:lipschitz-kkt} and Corollary~\ref{corollary:cont-bounded-inverse}, respectively.
    So let $L$ and $M$ be a Lipschitz constant and an upper bound for both matrix-valued functions.
    Then the expression above is bounded by
    \begin{align*}
        &\leq M \left\| \left( \partial_1 \kappa_1 \right)^{-1} -
            \left( \partial_1 \kappa_2 \right)^{-1}
            \right\| + M
        \left\| \partial_2 \kappa_1 - \partial_2 \kappa_2 \right\| \\
        &\leq 2 ML \| \eta_1 - \eta_2 \|,
    \end{align*}
    which shows $\partial z^*(\eta)$ is Lipschitz continuous with Lipschitz constant $2 ML$.
    A nearly identical proof can be used to show the gradient $\nabla J(\eta) = \gamma + \partial z^*(\eta)^T \cdot \nabla h(z^*(\eta))$ is Lipschitz continuous as well.
\end{proof}

We are now ready to prove the main result.

\begin{proof}[Proof of Theorem~\ref{thm:gradient-descent-full}.]% \paragraph{}
Our proof is essentially the same of that in~\cite{Pedregosa2016-rr}, except that we ignore error in the evaluation of the gradient.
Let $L$ be the Lipschitz constant of $\nabla J(\eta)$ and let $J^* = \inf_{\eta \in \mathcal H} J(\eta)$ (which is finite by continuity and compactness).
To show convergence, we will recursively apply the standard quadratic upper bound for functions with Lipschitz continuous gradients,
\begin{equation*}
    J(v) \leq J(u) + \nabla J(u)^T (v - u) + (L/2) \| u - v \|_2^2,
\end{equation*}
where $u, v \in \reals^K$, which can be found in~\cite{Poljak1987-wz}, for example.
We also use the fact that projections onto convex sets are nonexpansive~\cite{Ryu2016-qe},
\begin{equation*}
    \| \proj_{\mathcal H}(u) - \proj_{\mathcal H}(v) \|_2^2
    \leq (u - v)^T (\proj_{\mathcal H}(u) - \proj_{\mathcal H}(v)),
\end{equation*}
for $u, v \in \reals^K$.
In particular, for $u = \eta^i - \alpha \nabla J(\eta^i)$ and $v = \eta^i$, we obtain
\begin{equation}
\label{eq:non-expansive}
    \| \eta^{i+1} - \eta^i \|_2^2
    \leq -\alpha \nabla J(\eta^i)^T (\eta^{i+1} - \eta^i),
\end{equation}
Then we obtain the following inequality,
\begin{equation*}
\begin{aligned}
    J(\eta^{i+1})
    &\leq J(\eta^i) + \nabla J(\eta^i)^T (\eta^{i+1} - \eta^i) + (L/2) \| \eta^{i+1} - \eta^i \|_2^2 \\
    &\leq J(\eta^i) + \left( \frac{L}{2} - \frac{1}{\alpha} \right) \| \eta^{i+1} - \eta^i \|_2^2 & \textrm{by \eqref{eq:non-expansive}} \\
    &\leq J(\eta^i) - (L/2) \| \eta^{i+1} - \eta^i \|_2^2 & \textrm{since $\alpha \leq 1/L$}.
\end{aligned}
\end{equation*}
Applying this bound recursively, we obtain
\begin{equation*}
     J(\eta^{i+1}) \leq J(\eta^1) - \frac{L}{2} \sum_{j=1}^{i} \| \eta^{j+1} - \eta^j \|_2^2.
\end{equation*}
Since $J$ is a continuous function being minimized over a compact set, it has a finite lower bound $J^* \leq J(\eta)$ for all $\eta \in \mathcal H$.
This gives a finite upper bound on the series of square norms,
\begin{equation*}
    \sum_{j=1}^{i} \| \eta^{j+1} - \eta^j \|_2^2 \leq \frac{2(J(\eta^1) - J^*)}{L}.
\end{equation*}
In particular, this implies at least one of the $j$s in the summand must be such that $\| \eta^{j+1} - \eta^j \|_2^2 \leq 2(J(\eta^1) - J^*) / Li$, which proves the theorem.
\end{proof}

\section{Technical results on convex relaxation}

\subsection{Approximate strong duality formulation}
\label{appendix:strong-dual}

We restate Lemma~\ref{lemma:delta-strong-duality} here for convenience.

\begin{lemma*}
    Suppose $h$ is $L$-Lipschitz continuous.
    For all $\epsilon > 0$, there exists some $\delta > 0$ such that the $\delta$-strong duality formulation~\eqref{eq:mep-sd} is an $\epsilon$-approximation of~\eqref{eq:mep} over $\eta$, i.e., for all $(\eta, x, \lambda)$ feasible for~\eqref{eq:mep-sd},
    \begin{equation*}
        | h(x, \lambda) - h(x^*(\eta), \lambda^*(\eta)) | \leq \epsilon.
    \end{equation*}
\end{lemma*}

To prove Lemma~\ref{lemma:delta-strong-duality}, we first show that the feasibility for the $\delta$-strong duality problem implies that $(x, \lambda)$ is close to $(x^*(\eta), \lambda^*(\eta))$.
Then, it is straightforward to show main result.

\begin{lemma}
\label{lemma:epsilon-approx-is-close}
For all $\epsilon > 0$, there exists some $\delta > 0$ such that for all $\eta \in \mathcal H$,
\begin{equation}
\label{eq:approx-sd-constraints}
\begin{aligned}
    c(x) + b(\eta)^T \lambda + c^*(-A(\eta)^T \lambda) \leq \delta \\
    A(\eta) x \leq b(\eta) \\
    \lambda \geq 0
\end{aligned}
\end{equation}
implies
\begin{equation*}
    \max\left(
        \|x - x^*(\eta) \|_2,
        \| \lambda - \lambda^*(\eta) \|_2 \right) \leq \epsilon
\end{equation*}
\end{lemma}

\begin{proof}
Fix $\epsilon > 0$.
Now suppose for sake of contradiction that for all $\delta > 0$, there exists $(\eta, x, \lambda) \in \reals^N \times \reals^M$ such that the constraints in~\eqref{eq:approx-sd-constraints} are satisfied but $\| x - x^*(\eta) \|_2 \geq \epsilon$.
We can therefore build sequence of points $(\eta_1, x_1, \lambda_1), (\eta_2, x_2, \lambda_2), \ldots$ where $x_i$ is bounded away from $x_i^*(\eta_i)$ and $A(\eta_i) x_i \leq b(\eta_i)$, $\lambda_i \geq 0$, and $c(x_i) + b(\eta_i)^T \lambda_i + c^*(-A(\eta_i)^T \lambda_i) \leq 1/i$.
By continuity of $c$ and $c^\star$ (which follow from Lipschitz continuous gradients and strong convexity of $c$, respectively) and the compactness of the domain of $(x, \lambda)$, there exists a subsequence of this sequence that has a limiting point $(\hat \eta, \hat x, \hat \lambda)$ such that
\begin{equation*}
    c(\hat x) + b(\hat \eta)^T \hat \lambda + c^*(-A(\hat \eta)^T \hat \lambda) = 0.
\end{equation*}
Similarly, by the continuity of $x^*(\eta)$ (see Lemma~\ref{lemma:kkt-jac-invertible}), we know $\| \hat x - x^*(\hat \eta) \| \geq \epsilon$.
This implies that $x^*(\hat \eta)$ is not the unique solution to the dispatch model with parameter $\hat \eta$, which is a contradiction.
Using an identical line of reasoning, we can show the analogous result for $\| \lambda - \lambda^*(\eta)\|_2$, proving the claim.
\end{proof}

\begin{proof}[Proof of Lemma~\ref{lemma:delta-strong-duality}]
    First, lets show that the bound in Lemma~\ref{lemma:epsilon-approx-is-close} is uniform.
    Indeed

    Fix $\epsilon > 0$.
    By Lemma~\ref{lemma:epsilon-approx-is-close}, we can pick $\delta$ such that feasibility in~\eqref{eq:mep-sd} implies
    \begin{equation*}
        \max\left(
        \|x - x^*(\eta) \|_2,
        \| \lambda - \lambda^*(\eta) \|_2 \right) \leq \frac{\epsilon}{2L}
    \end{equation*}
    Now consider some $(\eta, x, \lambda)$ feasible for the $\delta$-strong duality formulation in~\eqref{eq:mep-sd}.
    By Lipschitz continuity,
    \begin{equation*}
        |h(x, \lambda) - h(x^*(\eta), \lambda^*(\eta))|
        \leq L \left(
            \| x - x^*(\eta) \|_2
            + \| \lambda - \lambda^*(\eta) \|_2
        \right)
        \leq \epsilon,
    \end{equation*}
    which proves the claim.
\end{proof}

\subsection{Tight integer relaxation}
\label{appendix:integer-envelope}

We restate the theorem here for convenience.

\begin{theorem}[Tight integer relaxation]
    Consider an integer constrained variable $\alpha \in \{ \alpha^{\min}, \alpha^{\max} \}$ and a box constrained variable $\beta \in [\beta^{\min}, \beta^{\max}]$ (which may or may not have integer constraints as well).
    Then for any $y \in \reals$, the equation $y = \alpha \beta$ holds if and only if
    \begin{equation*}
        y \in \envelope( \alpha \beta ).
    \end{equation*}
    As a consequence, if $\mathcal H = \mathcal H^{\discrete}$, then the feasible sets of~\eqref{eq:mep-sd} and its relaxation~\eqref{eq:mep-relax} are equal, i.e., the relaxation is tight.
\end{theorem}

\begin{proof}
    This result may be found in Theorem~3 of~\cite{Al-Khayyal1983-io}, but we prove it here for ease of access.
    By the definition of the convex and concave envelope, we already know that $y = \alpha \beta$ implies $y \in \envelope(\alpha \beta)$.
    So it suffices to show that $y \in \envelope(\alpha \beta)$ implies $y = \alpha \beta$.
    Recall that $y \in \envelope( \alpha \beta )$ if and only if,
    \begin{align*}
        y &\geq \alpha \beta^{\min} + \alpha^{\min} \beta - \alpha^{\min} \beta^{\min}, \\
        y &\geq \alpha^{\max} \beta + \alpha \beta^{\max} - \alpha^{\max} \beta^{\max}, \\
        y &\leq \alpha^{\min} \beta + \alpha \beta^{\max} - \alpha^{\min} \beta^{\max}, \\
        y &\leq \alpha \beta^{\min} + \alpha^{\max} \beta - \alpha^{\max} \beta^{\min}.
    \end{align*}
    Now first consider the case when $\alpha = \alpha^{\min}$.
    Then the first and third equations imply $y \geq \alpha \beta$ and $y \leq \alpha \beta$, respectively.
    So $y = \alpha \beta$.
    If $\alpha = \alpha^{\max}$, the the second and fourth equations imply $y \geq \alpha \beta$ and $y \leq \alpha \beta$, respectively, so $y = \alpha \beta$ in this case as well.
    Since $\alpha$ can only be either $\alpha^{\min}$ or $\alpha^{\max}$ by assumption, it holds that $y = \alpha \beta$ if and only if $y \in \envelope( \alpha \beta )$.
\end{proof}

\subsection{Approximate continuous relaxation}
\label{appendix:cont-envelope}

We restate Theorem~\ref{thm:cont-relax} here for convenience.

\begin{theorem}[Approximate continuous relaxation]
    % Fix some $\eta^0 \in \mathcal H^{\cont}$.
    For all $\epsilon > 0$, there exists some $\alpha > 0$ such that when
    % \begin{equation*}
    %     \max \left(
    %     \| \eta^{\min} - \eta^0 \|_{\infty},
    %     \| \eta^{\max} - \eta^0 \|_{\infty}
    %     \right) \leq \alpha,
    % \end{equation*}
    \begin{equation*}
        \| \eta^{\min} - \eta^{\max} \|_{\infty},
            \leq \alpha,
    \end{equation*}
    then~\eqref{eq:mep-relax} is a $\epsilon$-approximation of~\eqref{eq:mep-sd} over $\eta$.
    % i.e., for all $(\eta, x, \lambda, Y, V, \mu)$ feasible for \eqref{eq:mep-relax}, then there exists some $(x', \lambda')$ such that $(\eta, x', \lambda')$ is feasible for \eqref{eq:mep-sd} and
    % \begin{equation*}
    %     | h(x, \lambda) - h(x', \lambda') | \leq \epsilon.
    % \end{equation*}
\end{theorem}

To prove this, we first show the relaxed feasible set approximates the original feasible set in the limiting sense.

\begin{lemma}
    \label{lemma:shrinking-envelope}
    Fix $z^{\min}, z^{\max}, \eta^{\min}, \eta^{\max} \in \reals$.
    Then for all $z \in [z^{\min}, z^{\max}]$ and $\eta \in [\eta^{\min}, \eta^{\max}]$, it holds that
    \begin{equation}
        \sup_{y \in \envelope(z \eta)}
        | y - z \eta |
        \leq
        |z^{\max} - z^{\min}| |\eta^{max} - \eta^{\min}|.
    \end{equation}
\end{lemma}

\begin{proof}
    Fix $z \in [z^{\min}, z^{\max}]$ and $\eta \in [\eta^{\min}, \eta^{\max}]$.
    First note that the function $\psi(y) = |y - z \eta|$ is convex.
    Since the maximum of a convex function over a convex set is attained at an extreme point of the set, it suffices to consider the boundary of $\envelope(z \eta)$.
    At one of the lower convex envelopes we have,
    \begin{equation*}
    \begin{aligned}
        | \eta z^{\min} + \eta^{\min} z - \eta^{\min} z^{\min} - z \eta |
        = | (z - z^{\min}) (\eta - \eta^{\min}) |
        &\leq | z^{\max} - z^{\min} | | \eta^{\max} - \eta^{\min} |.
    \end{aligned}
    \end{equation*}
    The same line of reasoning can be used for the three other inequalities that define $\envelope(z \eta)$, proving the claim.
\end{proof}

By Assumption~\eqref{assume:bounded-domain}, the primal and dual dispatch solutions are assumed to lie in a compact domain.
Thus, we can assume they lie in some finite (albeit possibly large) box.
This implies that when $\| \eta^{\max} - \eta^{\min} \|_\infty$ is sufficiently small, the set $\envelope(z \eta)$ approaches the point $z \eta$.

Next, we state and prove a general result that we will later show includes Theorem~\ref{thm:cont-relax} as a special case.

\begin{lemma}
\label{lemma:bilinear-relaxation}

    Let $f : \reals^D \times \reals^D \rightarrow \reals$ and $g : \reals^D \times \reals^D \rightarrow \reals$ be convex, $L$-Lipschitz continuous functions and let $z^{\min}, z^{\max}, \eta^{\min}, \eta^{\max} \in \reals^D$.
    Assume that there exists $\alpha > 0$ such that for all $\eta \in [\eta^{\min}, \eta^{\max}]$, there exists $z \in [z^{\min}, z^{\max}]$ such that
    \begin{equation}
        \label{local-assume:strict-feasibility}
        g(z, \eta, \diag(\eta) z) < -\alpha.
    \end{equation}
    Now consider the optimization problem
    \begin{equation}
    \tag{O} \label{eq:bilinear-problem}
    \begin{array}{ll}
        \mathrm{minimize} \quad & f(z, \eta) \\[0.25em]
        \mathrm{subject\ to}
        & g(z, \eta, y) \leq 0 \\[0.25em]
        & y = \diag(\eta) z \\[0.25em]
        & \eta^{\min} \leq \eta \leq \eta^{\max} \\[0.25em]
        & z^{\min} \leq z \leq z^{\max},
    \end{array}
    \end{equation}
    where the variables are $z, \eta, y \in \reals^D$.
    Also consider the relaxed problem
    \begin{equation}
    \tag{R} \label{eq:relaxed-bilinear-problem}
    \begin{array}{lll}
        \mathrm{minimize} \quad & f(z, \eta) \\[0.25em]
        \mathrm{subject\ to}
        & g(z, \eta, y) \leq 0 \\[0.25em]
        & y_i \in \envelope(\eta_i z_i), \quad & i = 1, \ldots, D \\[0.25em]
        & \eta^{\min} \leq \eta \leq \eta^{\max} \\[0.25em]
        & z^{\min} \leq z \leq z^{\max},
    \end{array}
    \end{equation}
    where the variables are also $z, \eta, y \in \reals^D$.
    Then for all $z^{\min}, z^{\max} \in \reals^D$ and $\epsilon > 0$, there exists $\delta > 0$ such that,
    \begin{equation*}
        \| \eta^{\max} - \eta^{\min} \|_\infty \leq \delta
    \end{equation*}
    implies that~\eqref{eq:relaxed-bilinear-problem} is an $\epsilon$-approximation of~\eqref{eq:bilinear-problem} over $\eta$, i.e., for all $(z, \eta, y)$ feasible for~\eqref{eq:relaxed-bilinear-problem}, there exists $z'\in \reals^D$ such that $(z', \eta, \diag(\eta) z')$ is feasible for~\eqref{eq:bilinear-problem} and
    \begin{equation*}
        | f(z, \eta) - f(z', \eta) | \leq \epsilon.
    \end{equation*}
\end{lemma}

\begin{proof}
    Without loss of generality, assume $L \geq 1$, $\|z^{\max} - z^{\min}\|_2 \geq 1$, and $\epsilon < 1$.
    Fix $z^{\min}, z^{\max} \in \reals^D$ and $\epsilon > 0$ and suppose
    \begin{equation*}
        \| \eta^{\max} - \eta^{\min} \|_\infty \leq \delta
        = \frac{\alpha \epsilon}{2 L^2 \sqrt{D} \|z^{\max} - z^{\min} \|_\infty  \|z^{\max} - z^{\min} \|_2 }.
    \end{equation*}
    Now consider some point $v = (z, \eta, y)$ feasible for~\eqref{eq:relaxed-bilinear-problem}.
    Our proof will proceeds in three steps.

    \paragraph*{Step 1.}
    By~\eqref{local-assume:strict-feasibility}, there exists $\tilde z \in [z^{\min}, z^{\max}]$ such that $g(\tilde z, \eta, \diag(\eta) \tilde z) \leq -\alpha < 0$.
    Let $\tilde v = (\tilde z, \eta, \diag(\eta) \tilde z)$.
    The point $\tilde v$ is feasible for~\eqref{eq:bilinear-problem}, so it must be feasible for~\eqref{eq:relaxed-bilinear-problem} as well.
    Now choose $\hat v = (1 - \theta) v + \theta \tilde v$. where
    \begin{equation*}
        \theta = \frac{\epsilon}{L \| z^{\max} - z^{\min} \|_2 },
    \end{equation*}
    and note that $\hat \eta = \eta$.
    Since~\eqref{eq:relaxed-bilinear-problem} is convex, $\hat v$ is feasible for~\eqref{eq:relaxed-bilinear-problem} and
    \begin{equation*}
        g(\hat v) \leq (1-\theta) g(v) + \theta g(\tilde v)
        \leq - \theta \alpha.
    \end{equation*}

    \paragraph*{Step 2.}
    Let $v^* = (\hat z, \eta, \diag(\eta) \hat z)$.
    We claim $v^*$ is feasible for~\eqref{eq:bilinear-problem}.
    Of course, this point satisfies the second, third, and fourth constraints of~\eqref{eq:bilinear-problem} by construction.
    But since $\| \eta^{\max} - \eta^{\min} \|_\infty \leq \delta$, then by Lemma~\ref{lemma:shrinking-envelope} it must be that
    \begin{equation*}
        \| \hat v - v^* \|_2
        = \| \hat y - \diag(\eta) \hat z \|_2
        \leq \sqrt{D} \|z^{\max} - z^{\min} \|_\infty \delta.
    \end{equation*}
    So, since $g$ is $L$-Lipschitz, we have
    \begin{equation*}
    \begin{aligned}
        g(v^*)
        &\leq g(\hat v) + L \| \hat v - v^* \|_2 \\
        &\leq - \theta \alpha + L \sqrt{D} \|z^{\max} - z^{\min} \|_\infty \delta \\
        &\leq L \sqrt{D} \|z^{\max} - z^{\min} \|_\infty \delta
        - \frac{\alpha \epsilon}{L \| z^{\max} - z^{\min} \|_2} \\
        &\leq - \frac{\alpha \epsilon}{2 L \| z^{\max} - z^{\min} \|_2} \leq 0,
    \end{aligned}
    \end{equation*}
    which means the second constraint is satisfied, and hence $v^*$ is feasible for~\eqref{eq:bilinear-problem}.

    \paragraph*{Step 3.}
    Lastly, we claim that $|f(z, \eta) - f(z^*, \eta^*)| \leq \epsilon$.
    Recall that $\eta^* = \eta$ and $z^* = \hat z = (1-\theta) z + \theta \tilde z$.
    Since $f$ is $L$-Lipschitz, we have
    \begin{equation*}
    \begin{aligned}
        |f(z, \eta) - f(z^*, \eta^*)|
        &\leq L \| z - \hat z \|_2 \\
        &= L \theta \| z - \tilde z \|_2 \\
        &\leq L \theta \| z^{\max} - z^{\min} \|_2
        \leq \epsilon, \\
    \end{aligned}
    \end{equation*}
    which proves the claim.
\end{proof}

Finally, we can prove Theorem~\ref{thm:cont-relax} by showing it satisfies the conditions of Lemma~\eqref{lemma:bilinear-relaxation}.

\begin{proof}[Proof of Theorem~\ref{thm:cont-relax}]
    This is simply a special case of Lemma~\ref{lemma:bilinear-relaxation} with variable $z = (x, \lambda)$, objective $f(z, \eta) = \gamma^T \eta + h(x, \lambda)$, and constraint
    \begin{equation*}
        g(z, \eta, y) = \begin{bmatrix}
            c(x) - \mu - c^\star(-V \ones) + \delta \\
            Y \ones - b(\eta) \\
        \end{bmatrix},
    \end{equation*}
    (since each of $Y$, $V$, and $\mu$ are linear functions of $y = \diag(\eta) z$).
    The only other step is to show that that there exists $\alpha > 0$ such that $g(z, \eta, y) < -\alpha$.
    Of course, we can simply set $\alpha < \delta$ to show this for the first term in $g$.
    For the second term, this follows directly from Assumption~\ref{assume:feasible}, compactness assumptions on $\eta$ and $x$ (Assumption~\ref{assume:bounded-domain}).
\end{proof}

\newpage
\section{Experiment Details}
\label{appendix:pypsa}

We use the PyPSA-USA~\cite{Tehranchi2023-jo} dataset for all our experiments.
The software is freely available on Github at \url{https://github.com/pypsa/pypsa-usa}.
We specifically use commit \verb|065f864| from December 8th, 2023.
Using that version of the package should exactly reproduce the data as in the paper.
We generate data for the U.S.\ Western Interconnect for the full year of 2019 and use the default settings to cluster the network to the desired size.
The dataset includes fuel costs, capital, renewable generator maximum expansion capacities, a realistic electrical network with susceptances and capacities, battery locations and limits, and other data.

We scale all loads by 150\% in order to simulate future load growth.
Then, we decrease generation and transmission capacities by 20\% to model retirements and outages.
Existing transmission lines and batteries are allowed to be expanded by a factor of 10; expansion costs are given by the dataset itself.
Load is curtailable with a curtailment cost of \$500 per MWh.
Most solutions have no load curtailment.
Renewable expansion capacities are set to be 0.01\% of the PyPSA provided capacity limit.
This is because PyPSA uses technical potentials, i.e., the physical capacity of the land, ignoring zoning and existing construction.
All generators are given a maximum potential capacity of 50 GW for numerical stability reasons.
In practice, this limit is never encountered.
Finally, due to a bug in the PyPSA-USA code, we set battery capacity expansion costs manually.
The battery costs, from the Danish Energy Agency~\cite{Danish-Energy-Agency2018-qn}, are \$151,940/MWh of battery capacity, annualized over 25 years, and \$171,200/MW of maximum power output, annualized over 10 years.
The discount rate is set to 7.0\% and an annual fixed operations and maintenance cost of \$52,845/MW is also added.
Batteries with less than 1 hour of duration are ignored, since these are likely used for ancillary services in practice.
When solving cases for $T$ hours, all capital costs are scaled to $T/8760$ times their annualized cost in order to properly scale capital and operation costs.

\end{document}